\documentclass[pdflatex,sn-mathphys-num]{sn-jnl}

\usepackage{graphicx}%
\usepackage{multirow}%
\usepackage{amsmath,amssymb,amsfonts}%
\usepackage{amsthm}%
\usepackage{mathrsfs}%
\usepackage[title]{appendix}%
\usepackage{xcolor}%
\usepackage{textcomp}%
\usepackage{manyfoot}%
\usepackage{booktabs}%
\usepackage{algorithm}%
\usepackage{algorithmicx}%
\usepackage{algpseudocode}%
\usepackage{listings}%

\usepackage{url}

\theoremstyle{thmstyleone}%
\theoremstyle{thmstyletwo}%

\theoremstyle{thmstylethree}%
 \newtheorem{thm}{Theorem}[section]
 \newtheorem{cor}[thm]{Corollary}
 \newtheorem{lem}[thm]{Lemma}
 
 \theoremstyle{definition}
 
 \theoremstyle{remark}

 \numberwithin{equation}{section}

\begin{document}

\title[Minimality of the root functions]{Minimality of the root functions of Sturm-Liouville problems with a boundary condition depending linearly on an eigenparameter}


\author[1]{\fnm{Yagub} \sur{Aliyev}}\email{yaliyev@ada.edu.az}

\author*[2]{\fnm{Narmin} \sur{Aliyeva}}\email{nermin.aliyeva@idrak.edu.az}
\equalcont{These authors contributed equally to this work.}

\affil*[1]{\orgdiv{School of IT and Engineering}, \orgname{ADA University}, \orgaddress{\street{61 Ahmadbay Agha-Oglu Street}, \city{Baku }, \postcode{AZ1008}, \country{Azerbaijan}}}

\affil[2]{\orgdiv{Department of Differential Equations and Control Theory}, \orgname{Baku State University}, \orgaddress{\street{Academic Zahid Khalilov str. 23}, \city{Baku}, \postcode{AZ1148}, \country{Azerbaijan}}}


\abstract{We consider a Sturm--Liouville problem in which the spectral parameter appears linearly in one of the boundary conditions. The study focuses on the root functions of the problem, including eigenfunctions and associated functions corresponding to multiple eigenvalues. By employing the characteristic function of the boundary value problem, explicit representations are obtained for the biorthogonal system and for several special associated functions that play a crucial role in the spectral analysis. These representations allow previously established criteria for the basis and minimality properties of the system of root functions to be reformulated directly in terms of the characteristic function and its derivatives at the eigenvalues. As a consequence, the investigation of particular boundary value problems becomes considerably simpler. Several illustrative examples are analyzed to demonstrate the effectiveness of the proposed approach and to show its agreement with known results in the literature.}



\maketitle

\section{Introduction}\label{sec1}

We consider the following Sturm-Liouville problem
\begin{equation}
  -y^{\prime \prime }+q(x)y=\lambda y,\ 0<x<1,  \label{(1.1)} 
\end{equation}
\begin{equation}
y(0)\cos \beta =y^{\prime }(0)\sin \beta ,\ 0\leq \beta <\pi ,
\label{(1.2)}
\end{equation}
\begin{equation}
(a\lambda + b)y(1) = (c\lambda + d)y{'}(1). \label{(1.3)}
\end{equation}
Here $a,b,c,d$ are reals and $ad-bc<0$. Note that the spectral parameter $\lambda $ is contained not just in the equation but also in the boundary condition. For simplicity $q(x)$ is assumed to be a continuous real valued function over the interval $[0,1]$. In our previous paper \cite{aliyev5} we studied this problem and gave a detailed list of previous works on the problem. In particular, in \cite{aliyev5} we obtained necessary and sufficient conditions for the basis properties of the root functions with one function removed. In \cite{aliyev5} these necessary and sufficient conditions were expressed in terms of special associated functions $y_{k+1}^*$, $y_{k+1}^\#$, and $y_{k+2}^\#$. In the current paper these conditions and the biorthogonal system of the root functions  are expressed directly through the characteristic function. The involvement of the characteristic function considerably simplifies the study of particular boundary value problems. To demonstrate this, we looked at the same examples which were considered in \cite{aliyev5} and solved them in a new way.

Similar problems were  studied in \cite{fulton,kerimov1, aliyev1, aliyev2, zaliyev1,benedek, li, guliyev, kerimov, mois, maris, moller, olgar}. In more general form of polynomial dependence on $\lambda$ was studied in \cite{rus} and \cite{shkalikov1}. The results of the current paper are in perfect agreement with the results in \cite{shkalikov2} (Theorem 3), where (\ref{(1.2)}) is replaced by $y'(0)=0$. As was pointed out in \cite{shkalikov2}, the cases where the system of root functions with one function removed does not form a minimal system are exceptional. The current paper explicitly determines these cases.

More general problems were studied in \cite{shkalikov2,kerimov0}, where the parameter $\lambda$ participated both in (\ref{(1.2)}) and (\ref{(1.3)}).

Following \cite{binding1}, we assume that the eigenvalues
$\{\lambda_n\}_{n\ge0}$ of problem (\ref{(1.1)})--(\ref{(1.3)})
form an infinite sequence whose only accumulation point is $+\infty$.
The eigenvalues are indexed according to non-decreasing real parts and
are counted with their algebraic multiplicities. Furthermore, exactly
one of the following situations occurs:

\begin{enumerate}
\item[(i)] all eigenvalues are real and simple;

\item[(ii)] all eigenvalues are real, with exactly one double eigenvalue,
while all remaining eigenvalues are simple;

\item[(iii)] all eigenvalues are real, with exactly one triple
eigenvalue, and all other eigenvalues being simple;

\item[(iv)] all eigenvalues are simple, and all eigenvalues are real
except for a single pair of complex conjugate eigenvalues.
\end{enumerate}
In \cite[Theorem~2.2]{binding2} it is shown that asymptotically the eigenvalues are
\begin{equation}
\lambda_n =
\begin{cases}
\left(n-1\right)^2\pi^2 + O(1), & \text{if } \beta \neq 0,\\[6pt]
\left(n-\dfrac{1}{2}\right)^2\pi^2 + O(1), & \text{if } \beta = 0.
\end{cases}\label{(1.4)}
\end{equation}

In \cite{aslanova, aslanova1} similar boundary value problems with operator coefficients were studied. In \cite{chein, liu} numerical methods  to calculate the eigenvalues were developed. 

\section{The eigenfunctions $y_{n}$.}\label{sec2}

Let $y(x,\lambda)$ be the solution of the boundary value problem  
\begin{equation}
-y^{\prime \prime }+q(x)y=\lambda y,  \label{(2.1)}
\end{equation}
\begin{equation}
y^{\prime }(0)=\cos \beta ,\ \   y(0)=\sin \beta \label{(2.2)}
\end{equation}
We defined the characteristic function as 
\begin{equation}
\omega(\lambda) = (a\lambda + b)y(1,\lambda) - (c\lambda + d)y{'}(1,\lambda).\label{(2.3)}
\end{equation}
By (\ref{(1.1)})-(\ref{(1.3)}), if \(\omega(\lambda_n)=0\), then $y(x,\lambda_n)=y_{n}$ is an eigenfunction corresponding to $\lambda=\lambda_n$, i.e.
\begin{equation}
-y_{n}^{\prime \prime }+q(x)y_{n}=\lambda _{n}y_{n},  \label{(2.4)}
\end{equation}
\begin{equation}
y_{n}^{\prime }(0)\sin \beta =y_{n}(0)\cos \beta ,  \label{(2.5)}
\end{equation}
\begin{equation}
(a\lambda _{n}+b)y_{n}(1)=(c\lambda _{n}+d)y^{\prime }_{n}(1).  \label{(2.6)}
\end{equation}
Note that  $y(x,\lambda)$ is normalized and therefore    $\lim_{\lambda \to \lambda_n}y(x,\lambda)=y_n$ uniformly over $0<x<1$. Similarly, let $y(x,\mu)$ be the solution of (\ref{(2.1)}) and (\ref{(2.2)}) corresponding to $\mu$. 
By Lagrange's identity
$$
\frac{d}{dx}(y(x,\lambda)\overline{y^{\prime }(x,\mu)}-y^{\prime
}(x,\lambda)\overline{y(x,\mu)})= (\lambda - \overline
{\mu})y(x,\lambda)\overline{y(x,\mu)}.
$$
By integrating both sides of this equality from 0 to 1, we obtain
\begin{equation}
(\lambda - \overline
{\mu})\left(y(\cdot,\lambda),y(\cdot,\mu)\right)={\left. {(y(x,\lambda)\overline{y^{\prime }(x,\mu)}-y^{\prime
}(x,\lambda)\overline{y(x,\mu)}) } \right|}_{0}^{1},
\label{(2.7)} \end{equation}
where $(\cdot ,\cdot )$ as usual means the inner product in $L_{2}(0,1)$.
Simplifying this equality  using (\ref{(2.2)}) and (\ref{(2.3)}) we obtain that if $\lambda \ne\overline
{\mu}$ and $\lambda,
{\mu}\ne-\frac{d}{c}$, then
$$
\left(y(\cdot,\lambda), y(\cdot,\mu)\right)
= -(ad - bc)\frac{y(1,\lambda)\overline{y(1,\mu)}}{(c\lambda + d)(c\overline{\mu} + d)}
$$
\begin{equation}
- \frac{\omega(\overline{\mu})}{\lambda - \overline{\mu}} \cdot \frac{y(1,\lambda)}{c\overline{\mu} + d}
+ \frac{\omega(\lambda)}{\lambda - \overline{\mu}} \cdot \frac{\overline{y(1,\mu)}}{c\lambda + d}.\label{(2.8)}
\end{equation}
In a similar way we can show that 
if $\lambda \ne\overline
{\mu}$ and $\lambda,
{\mu}\ne-\frac{b}{a}$, then
\[
\left(y(\cdot,\lambda), y(\cdot,\mu)\right)
= -(ad - bc)\frac{y^\prime(1,\lambda)\overline{y^\prime(1,\mu)}}{(a\lambda + b)(a\overline{\mu} + b)}
\]
\begin{equation}
- \frac{\omega(\overline{\mu})}{\lambda - \overline{\mu}} \cdot \frac{y^\prime(1,\lambda)}{a\overline{\mu} + b}
+ \frac{\omega(\lambda)}{\lambda - \overline{\mu}} \cdot \frac{\overline{y^\prime(1,\mu)}}{a\lambda + b}.\label{(2.9)}
\end{equation}
If  $\lambda=\lambda_n=-\frac{d}{c}$ in (\ref{(2.7)}), then  $y(1,\lambda)=y_n(1)=0$ and therefore  $$
(\lambda_n - \overline
{\mu})\left(y_n(\cdot),y(\cdot,\mu)\right)= -y_n^{\prime
}(1)\overline{y(1,\mu)},
 $$
from which we obtain 
\begin{equation}
\left(y_n(\cdot), y(\cdot,\mu)\right)
= -(ad - bc)\frac{y_n^\prime(1)\overline{y(1,\mu)}}{(a\lambda_n + b)(c\overline{\mu} + d)}. \label{(2.10)}
\end{equation}
Let us introduce the notation 
\begin{equation}
\mathfrak{A}(y_n) =
\begin{cases}
\displaystyle  
\dfrac{y_n(1)}{c\lambda_n + d}, 
& \text{if } \lambda_n \neq -\dfrac{d}{c}, \\[10pt]
\displaystyle \ 
\dfrac{y_n'(1)}{a\lambda_n + b}, 
& \text{if } \lambda_n = -\dfrac{d}{c}.
\end{cases}\label{(2.11)}
\end{equation}
\begin{lem}
If $\lambda _{n}\ne \overline{\lambda _{m}}$, then 
\begin{equation}
\left( y_{n},y_{m} \right) = - (ad-bc)\mathfrak{A}(y_n) \overline{\mathfrak{A}(y_m)} .\label{(2.12)}
\end{equation}
\end{lem}
\begin{proof}
If $\lambda_n,\lambda_m\neq-\frac{d}{c}$, then we substitute  $\lambda=\lambda_n$ and $\mu=\lambda_m$ in (\ref{(2.8)}) and use the fact that $\omega(\lambda_n)=\omega(\lambda_m)=0$.
If \(\lambda_{n}= - \frac{d}{c} \neq \lambda_{m}\), then we substitute $\mu=\lambda_m$ in (\ref{(2.10)}). The remaining case  \(\lambda_{m}= - \frac{d}{c} \neq \lambda_{n}\) is considered similarly.
\end{proof}
\begin{cor}
If $\lambda _{r}$ is a non-real eigenvalue, then
\begin{equation} 
  {\left\| {y_{r} }\right\|}^{2}_{2} = -(ad - bc) {\left| {\mathfrak{A}(y_{r})} \right|^2}\label{(2.13)}
 \end{equation}
\end{cor}

\begin{lem} If $\lambda _{n}$ is real 
\begin{equation}
{\left\| {y_{n} } \right\|}^{2}_{2}=-(ad-bc)\mathfrak{A}^2(y_{n})+\mathfrak{A}(y_{n})\omega^\prime(\lambda_n).
\label{(2.14)}
\end{equation}
\end{lem}
\begin{proof}
If $\lambda_n\neq-\frac{d}{c}$, then we substitute  $\mu=\lambda_n$ in (\ref{(2.8)}) and obtain  
\begin{equation}
\left(y(\cdot,\lambda), y_n(\cdot)\right)
= -(ad - bc)\frac{y(1,\lambda)}{c\lambda + d}\mathfrak{A}(y_{n})
+ \frac{\omega(\lambda)}{\lambda - \lambda_n} \cdot \frac{{y_n(1)}}{c\lambda + d}.\label{(2.15)}
\end{equation}
By tending $\lambda\to\lambda_n$ we obtain (\ref{(2.14)}) in this case.
If \(\lambda_{n}= - \frac{d}{c}\), then we write (\ref{(2.10)}) as
\[
\left(y_n(\cdot), y(\cdot,\mu)\right)
= -(ad - bc)\mathfrak{A}(y_{n})\frac{{\omega(\overline{\mu})+(c\overline{\mu} + d)}\overline{y^\prime(1,\mu)}}{(a\overline{\mu} + b)(c\overline{\mu} + d)}
\]
\begin{equation}
= -(ad - bc)\mathfrak{A}(y_{n})\frac{\overline{y^\prime(1,\mu)}}{a\overline{\mu} + b}+\frac{y_n^\prime(1)}{a\overline{\mu}+b}\frac{\omega(\overline{\mu})}{\overline{\mu}-\lambda_n}. \label{(2.16)}
\end{equation}
By tending $\mu\to\lambda_n$ we obtain (\ref{(2.14)}) in this case.
\end{proof}

\begin{cor}
If $\lambda _{k}$ is a multiple eigenvalue, then 
\begin{equation}
{\left\| {y_{k}} \right\|}^{2}_{2}=-(ad-bc)\mathfrak{A}^2(y_{k}).
\label{(2.17)}
\end{equation}
\end{cor}
\begin{lem}
If $\lambda_r$ and $\lambda_s = \overline{\lambda_r}$ ($s:=r+1$) form a conjugate pair of non-real eigenvalues, then

\begin{equation}
(y_{s},y_{r})=-(ad-bc)\mathfrak{A}^2(y_{s})+\mathfrak{A}(y_{s})\omega^\prime(\lambda_s).\label{(2.18)}
\end{equation}
\end{lem}

\begin{proof}
If we substitute $\mu=\lambda_r$ in (\ref{(2.8)}), then we obtain 
$$
\left(y(\cdot,\lambda), y_r(\cdot)\right)
= -(ad - bc)\frac{y(1,\lambda)}{c\lambda + d}\overline{\mathfrak{A}(y_r)}
+ \frac{\omega(\lambda)}{\lambda - \overline{\lambda_r}} \cdot \frac{\overline{y_r(1)}}{c\lambda + d}.
$$
By tending $\lambda\to\lambda_s$ we obtain (\ref{(2.17)}).
\end{proof}

\section{The first associated function \(y_{k + 1}\).}
If \(\lambda_{k}\) is a multiple  eigenvalue
(\(\lambda_{k} = \lambda_{k + 1}\le\lambda_{k + 2}\)), then the associated function
\(y_{k + 1}\) of the eigenfunction \(y_{k}\) is defined by
\begin{equation}
    {- y''_{k + 1}} + q(x)y_{k + 1} = \lambda_{k}y_{k + 1} + y_{k}, \label{(3.1)}
\end{equation}
\begin{equation}
y_{k + 1}(0)\cos\beta = y'_{k + 1}(0)\sin\beta, \label{(3.2)}
\end{equation}
\begin{equation}
\left( a\lambda_{k} + b \right)y_{k + 1}(1) + ay_{k}(1) = \left( c\lambda_{k} + d \right)y'_{k + 1}(1) + cy'_{k}(1). \label{(3.3)}
\end{equation}
By differentiating (\ref{(1.1)}), (\ref{(1.2)}), and (\ref{(2.3)}) with respect to $\lambda$ we obtain 
\begin{equation}
  -y_\lambda^{\prime \prime }+q(x)y_\lambda=\lambda y_\lambda+y,\ 0<x<1,  \label{(3.4)} 
\end{equation}
\begin{equation}
y_\lambda(0)\cos \beta =y_\lambda^{\prime }(0)\sin \beta ,\ 0\leq \beta <\pi ,
\label{(3.5)}
\end{equation}
\begin{equation}
 \omega^\prime(\lambda) = (a\lambda + b)y_\lambda(1,\lambda)+ay(1,\lambda) - (c\lambda + d)y_\lambda{'}(1,\lambda)-cy^\prime(1,\lambda), \label{(3.6)}   
\end{equation}
respectively. By (\ref{(3.1)})-(\ref{(3.6)}), \(\omega(\lambda_k)=\omega^\prime(\lambda_k)=0\), and $$\lim_{\lambda\to\lambda_k}y_\lambda(x,\lambda)=\tilde{y}_{k+1},$$uniformly over $x\in[0,1]$, where $\tilde{y}_{k+1}$ is a fixed associated function. Note that $y_{k + 1}$ is not unique and $y_{k + 1}=\tilde{y}_{k+1}+Cy_k$ for some constant $C$.
Let us define 
\begin{equation}
\mathfrak{A}(y_{k+1}) =
\begin{cases}
\displaystyle  
\frac{y_{k + 1}(1)}{c\lambda_{k}+d} - \frac{cy_{k}(1)}{{(c\lambda_{k}+d)}^{2}}, 
& \text{if } \lambda_k \neq -\dfrac{d}{c}, \\[10pt]
\displaystyle \ 
\frac{y'_{k + 1}(1)}{a\lambda_{k}+b} - \frac{ay'_{k}(1)}{{(a\lambda_{k}+b)}^{2}}, 
& \text{if } \lambda_k = -\dfrac{d}{c}.
\end{cases}\label{(3.7)}
\end{equation}
We define $\mathfrak{A}(\tilde{y}_{k+1})$ similarly. Since in the multiple eigenvalue cases there are no non-real eigenvalues, we will not write complex conjugates in the following formulas.
\begin{lem} Suppose that \(\lambda_{k}\) is a multiple  eigenvalue and \(\lambda_{n}\) is a simple  eigenvalue \((\lambda_{k}=\lambda_{k+1}\ne\lambda_{n})\). Then 
\begin{equation}
\left( y_{k + 1},y_{n} \right) = - (ad - bc)\mathfrak{A}(y_{k+1})\mathfrak{A}(y_{n}). \label{(3.8)}
\end{equation}
\end{lem}

\begin{proof}
If $\lambda_k \neq -\dfrac{d}{c}\neq\lambda_n$, then we substitute $\mu=\lambda_n$ in (\ref{(2.8)}) and obtain 
\begin{equation}
\left(y(\cdot,\lambda), y_n(\cdot)\right)
= -(ad - bc)\frac{y(1,\lambda)}{c\lambda + d}\mathfrak{A}(y_{n})+ \frac{\omega(\lambda)}{\lambda - \lambda_n} \cdot \frac{y_n(1)}{c\lambda + d}.\label{(3.9)}
\end{equation}
By differentiating with respect to $\lambda$ and substituting $\lambda=\lambda_k$ we obtain 
\begin{equation}
\left( \tilde{y}_{k + 1},y_{n} \right) = - (ad - bc)\mathfrak{A}(\tilde{y}_{k+1})\mathfrak{A}(y_{n}). \label{(3.10)}
\end{equation}
By Lemma 2.1,
\begin{equation}
\left({y}_{k},y_{n} \right) = - (ad - bc)\mathfrak{A}({y}_{k})\mathfrak{A}(y_{n}). \label{(3.11)}
\end{equation}
By adding (\ref{(3.10)}) to  (\ref{(3.11)}) multiplied by $C$ we obtain (\ref{(3.8)}) in this case.

If $\lambda_k \neq -\dfrac{d}{c}=\lambda_n$, then we differentiate (\ref{(2.10)}) with respect to $\mu$ and obtain
\begin{equation}
\left(y_\mu(\cdot), y_n(\cdot)\right)
= -(ad - bc)\left(\frac{{y_\mu(1,\mu)}}{c\mu + d}-\frac{{cy(1,\mu)}}{(c\mu + d)^2}\right)\mathfrak{A}({y}_{n}). \label{(3.12)}
\end{equation}
If we substitute $\mu=\lambda_k$ in (\ref{(3.12)}), then we obtain (\ref{(3.10)}) and then (\ref{(3.8)}) in this case.

Finally, if $\lambda_k = -\dfrac{d}{c}\neq\lambda_n$, then we need to consider two cases: $\lambda_n\neq-\frac{b}{a}$ and $\lambda_n=-\frac{b}{a}$.
First note that by (\ref{(2.3)}) 
\begin{equation*}
    \frac{y{'}(1,\mu)}{a\mu + b}=\frac{y(1,\mu)}{c\mu + d}-\frac{\omega(\mu)}{(a\mu + b)(c\mu + d)},
\end{equation*}
which we substitute in (\ref{(2.9)}) to obtain 
\[
\left(y(\cdot,\lambda), y(\cdot,\mu)\right)
= -(ad - bc)\frac{y^\prime(1,\lambda)}{a\lambda + b}\left(\frac{y(1,\mu)}{c\mu + d}-\frac{\omega(\mu)}{(a\mu + b)(c\mu + d)} \right)
\]
\begin{equation}
- \frac{\omega({\mu})}{\lambda - {\mu}} \cdot \frac{y^\prime(1,\lambda)}{a{\mu} + b}
+ \frac{\omega(\lambda)}{\lambda - {\mu}} \cdot \frac{{y^\prime(1,\mu)}}{a\lambda + b}.\label{(3.13)}
\end{equation}
If  $\lambda_n\neq-\frac{b}{a}$, then we put in the last equality $\mu=\lambda_n$ and obtain 
\[
\left(y(\cdot,\lambda), y_n(\cdot)\right)
= -(ad - bc)\frac{y^\prime(1,\lambda)}{a\lambda + b}\frac{y_n(1)}{c\lambda_n + d}+\frac{\omega(\lambda)y^\prime_n(1)}{(\lambda-\lambda_n)(a\lambda + b)}.
\]
By taking derivative of this equality with respect to $\lambda$ we obtain 
\[
\left(y_\lambda(\cdot,\lambda), y_n(\cdot)\right)
= -(ad - bc)\left(\frac{y_\lambda^\prime(1,\lambda)}{a\lambda + b}-\frac{ay^\prime(1,\lambda)}{(a\lambda + b)^2}\right)\frac{y_n(1)}{c\lambda_n + d}
\]
\begin{equation*}
    +\left(\frac{\omega^\prime(\lambda)}{(\lambda-\lambda_n)(a\lambda + b)}-\frac{\omega(\lambda)}{(\lambda-\lambda_n)^2(a\lambda + b)}-\frac{a\omega(\lambda)}{(\lambda-\lambda_n)(a\lambda + b)^2}\right)y_n^\prime(1).
\end{equation*}
By substituting $\lambda=\lambda_k$ and noting that $\omega(\lambda_k)=\omega^\prime(\lambda_k)=0 $  we obtain again (\ref{(3.10)}) and then using (\ref{(3.11)}) we prove (\ref{(3.8)}).
If  $\lambda_n=-\frac{b}{a}$, then tending $\mu\to\lambda_n$ in (\ref{(3.13)}) and noting $y^\prime(1,\mu)\to y_n^\prime(1)=0$ gives 
\begin{equation}
   \left(y(\cdot,\lambda), y_n(\cdot)\right)
= -(ad - bc)\frac{y^\prime(1,\lambda)}{a\lambda + b}\frac{y_n(1)}{c\lambda_n + d} \label{(3.14)}
\end{equation}
Taking the derivative with respect to $\lambda$ and then substituting $\lambda=\lambda_k$ gives (\ref{(3.10)}) and finally (\ref{(3.8)}) in this case.
\end{proof}
\begin{lem} Suppose that  $\lambda _{k}$ is a multiple eigenvalue \((\lambda_{k}=\lambda_{k+1} \le\lambda_{k+2})\).
Then 
\begin{equation}
\left( y_{k + 1},y_{k} \right) = - (ad - bc)\mathfrak{A}(y_{k+1})\mathfrak{A}(y_{k})+\mathfrak{A}(y_k)\frac{\omega^{\prime\prime}(\lambda_k)}{2}. \label{(3.16)}
\end{equation}
\end{lem}

\begin{proof}
If $\lambda_k\neq-\frac{d}{c}$, then we substitute $\mu=\lambda_k$ in (\ref{(2.8)}) and obtain 
\begin{equation}
\left(y(\cdot,\lambda), y_k(\cdot)\right)
= -(ad - bc)\frac{y(1,\lambda){y_k(1)}}{(c\lambda + d)(c{\lambda_k} + d)}+ \frac{\omega(\lambda)}{\lambda - {\lambda_k}} \cdot \frac{{y_k(1)}}{c\lambda + d}.\label{(3.17)}
\end{equation}
By taking the derivative with respect to $\lambda$ we obtain
$$
\left(y_\lambda(\cdot,\lambda), y_k(\cdot)\right)
= -(ad - bc)\left(\frac{y_\lambda{(1,\lambda)}}{c\lambda + d}-\frac{cy{(1,\lambda)}}{(c\lambda + d)^2}\right)\frac{{y_k(1)}}{c{\lambda_k} + d}
$$
\begin{equation}
    +\left(\frac{\omega^\prime(\lambda)}{(\lambda-\lambda_k)(c\lambda + d)}-\frac{\omega(\lambda)}{(\lambda-\lambda_k)^2(c\lambda + d)}-\frac{c\omega(\lambda)}{(\lambda-\lambda_k)(c\lambda + d)^2}\right)y_k(1).\label{(3.18)}
\end{equation}
By tending $\lambda\to\lambda_k$ and noting $\omega^\prime(\lambda_k)=0$ we obtain 
\begin{equation}
\left( \tilde{y}_{k + 1},y_{k} \right) = - (ad - bc)\mathfrak{A}(\tilde{y}_{k+1})\mathfrak{A}(y_{k})+\mathfrak{A}(y_k)\frac{\omega^{\prime\prime}(\lambda_k)}{2}. \label{(3.19)}
\end{equation}
By (\ref{(2.14)}) 
\begin{equation}
\left( {y}_{k},y_{k} \right) = - (ad - bc)(\mathfrak{A}(y_{k}))^2. \label{(3.20)}
\end{equation}
By multiplying (\ref{(3.20)}) to $C$ and adding to (\ref{(3.19)}) we obtain (\ref{(3.16)}) in this case.
If $\lambda_k=-\frac{d}{c}$, then we substitute $\mu=\lambda_k$ in (\ref{(2.9)}) and obtain 
\begin{equation}
\left(y(\cdot,\lambda), y_k(\cdot)\right)
= -(ad - bc)\frac{y^\prime(1,\lambda){y^\prime_k(1)}}{(a\lambda + b)(a{\lambda_k} + b)}+ \frac{\omega(\lambda)}{\lambda - {\lambda_k}} \cdot \frac{{y^\prime_k(1)}}{a\lambda + b}.\label{(3.21)}
\end{equation}
By taking the derivative with respect to $\lambda$ we obtain
$$
\left(y_\lambda(\cdot,\lambda), y_k(\cdot)\right)
= -(ad - bc)\left(\frac{y^\prime_\lambda{(1,\lambda)}}{a\lambda + b}-\frac{ay^\prime{(1,\lambda)}}{(a\lambda + b)^2}\right)\frac{{y^\prime_k(1)}}{a{\lambda_k} + b}
$$
\begin{equation}
    +\left(\frac{\omega^\prime(\lambda)}{(\lambda-\lambda_k)(a\lambda + b)}-\frac{\omega(\lambda)}{(\lambda-\lambda_k)^2(a\lambda + b)}-\frac{a\omega(\lambda)}{(\lambda-\lambda_k)(a\lambda + b)^2}\right)y^\prime_k(1).\label{(3.22)}
\end{equation}
By tending $\lambda\to\lambda_k$ and noting $\omega^\prime(\lambda_k)=0$ we obtain (\ref{(3.19)}) in this case. Again using (\ref{(3.20)}) we prove (\ref{(3.16)}) in this case.
\end{proof}

Let us define 
$\mathfrak{A}(\hat{y}_{k+1})$ as in (\ref{(3.7)}).

\begin{lem} Suppose that  $\lambda _{k}$ is a multiple eigenvalue \((\lambda_{k}=\lambda_{k+1} \le\lambda_{k+2})\).
Then 
\begin{equation}
{\left\| {y_{k+1}} \right\|}^{2}_{2} = - (ad - bc)(\mathfrak{A}(y_{k+1}))^2+\mathfrak{A}(\hat{y}_{k+1})\frac{\omega^{\prime\prime}(\lambda_k)}{2}+\mathfrak{A}(y_k)\frac{\omega^{\prime\prime\prime}(\lambda_k)}{6}, \label{(3.24)}
\end{equation}
where $\hat{y}_{k+1}={y}_{k+1}+Cy_k$.
\end{lem}

\begin{proof}
If $\lambda_k\neq-\frac{d}{c}$, then we differentiate (\ref{(2.8)}) with respect to $\lambda$ and substitute $\lambda=\lambda_k$ to obtain 
$$
\left(\tilde{y}_{k+1}(\cdot), y(\cdot,\mu)\right)
= -(ad - bc)\mathfrak{A}(\tilde{y}_{k+1})\frac{{y(1,\mu)}}{c{\mu} + d}
$$
\begin{equation}
- \frac{\omega{(\mu})}{{c\mu+d}} \cdot \frac{y_k(1)}{\lambda_k-{\mu}}\left(\frac{{\tilde{y}_{k+1}(1)}}{y_k(1)}-\frac{1}{\lambda_k-{\mu}} \right)
\label{(3.25)}
\end{equation}
By differentiating this with respect to $\mu$ and tending $\mu\to\lambda_k$ we obtain
\begin{equation}
\left(\tilde{y}_{k+1}(\cdot),\tilde{y}_{k+1}(\cdot)\right)
= -(ad - bc)(\mathfrak{A}(\tilde{y}_{k+1}))^2+\mathfrak{A}(\tilde{y}_{k+1})\frac{\omega^{\prime\prime}(\lambda_k)}{2}+\mathfrak{A}(y_k)\frac{\omega^{\prime\prime\prime}(\lambda_k)}{6}.\label{(3.26)}
\end{equation}
By adding to the last equality the equality (\ref{(3.19)}) multiplied by $C$ we obtain 
$$
\left(\tilde{y}_{k+1}(\cdot),{y}_{k+1}(\cdot)\right)
= -(ad - bc)\mathfrak{A}(\tilde{y}_{k+1})\mathfrak{A}({y}_{k+1})
$$
\begin{equation}
+\mathfrak{A}({y}_{k+1})\frac{\omega^{\prime\prime}(\lambda_k)}{2}+\mathfrak{A}(y_k)\frac{\omega^{\prime\prime\prime}(\lambda_k)}{6}.\label{(3.27)}
\end{equation}
Finally, by adding the equality (\ref{(3.19)}) multiplied by $C$  to the last equality we obtain (\ref{(3.24)}) in this case. If $\lambda_k=-\frac{d}{c}$, then we do same process with (\ref{(2.9)}).
\end{proof}

\section{The second associated function \(y_{k + 2}\).}
If \(\lambda_{k}\) is a triple  eigenvalue
(\(\lambda_{k} = \lambda_{k + 1}=\lambda_{k + 2}\)), then the associated function
\(y_{k + 2}\) of the associated function \(y_{k+1}\) is defined by
\begin{equation}
 -y_{k+2}^{\prime \prime }+q(x)y_{k+2}=\lambda _{k}y_{k+2}+y_{k+1}, \label{(4.1)}  
\end{equation}
\begin{equation}
   y_{k+2}(0)\cos \beta  =y_{k+2}^{\prime }(0)\sin \beta,     \label{(4.2)}
\end{equation}
\begin{equation}
 \left( a\lambda_{k} + b \right)y_{k + 2}(1) + ay_{k + 1}(1) = \left( c\lambda_{k} + d \right)y'_{k + 2}(1) + cy'_{k + 1}(1).  \label{(4.3)}   
\end{equation}
By differentiating (\ref{(3.4)}), (\ref{(3.5)}), and (\ref{(3.6)}) with respect to $\lambda$ we obtain 
\begin{equation}
  -y_{\lambda\lambda}^{\prime \prime }+q(x)y_{\lambda\lambda}=\lambda y_{\lambda\lambda}+2y_{\lambda},\ 0<x<1,  \label{(4.4)} 
\end{equation}
\begin{equation}
y_{\lambda\lambda}(0)\cos \beta =y_{\lambda\lambda}^{\prime }(0)\sin \beta ,\ 0\leq \beta <\pi ,
\label{(4.5)}
\end{equation}
\begin{equation}
 \omega^{\prime\prime}(\lambda) = (a\lambda + b)y_{\lambda\lambda}(1,\lambda)+2ay_\lambda(1,\lambda) - (c\lambda + d)y_{\lambda\lambda}{'}(1,\lambda)-2cy_\lambda^\prime(1,\lambda),\label{(4.6)}   
\end{equation}
respectively. By (\ref{(4.1)})-(\ref{(4.6)}), \(\omega(\lambda_k)=\omega^\prime(\lambda_k)=\omega^{\prime\prime}(\lambda_k)=0\), and 
$$\lim_{\lambda\to\lambda_k}y_{\lambda\lambda}(x,\lambda)=2\tilde{y}_{k+2},$$uniformly over $x\in[0,1]$, where $\tilde{y}_{k+2}$ is a fixed associated function. Note that $y_{k + 2}$ is not unique and $y_{k + 2}=\tilde{y}_{k+2}+C\tilde{y}_{k+1}+Dy_k$ for some constant $D$. In addition, if $\check{C}y_k$ is added to the associated function $y_{k+1}$ for some constant $\check{C}$, then $\check{C}y_{k+1}$ is added to the associated function $y_{k+2}$. We also define 
\begin{equation}
\mathfrak{A}(y_{k+2}) =
\begin{cases}
\displaystyle  
\frac{y_{k + 2}(1)}{c\lambda_{k} + d} - \frac{cy_{k+1}(1)}{\left( c\lambda_{k} + d \right)^{2}}
+\frac{c^2y_{k}(1)}{\left( c\lambda_{k} + d \right)^{3}}, 
&  \text{if }  \lambda_k \neq -\dfrac{d}{c}, \\[10pt]
\displaystyle \ 
\frac{y'_{k + 2}(1)}{a\lambda_{k} + b} - \frac{ay'_{k+1}(1)}{\left( a\lambda_{k} + b \right)^{2}}
+\frac{a^2y'_{k}(1)}{\left( a\lambda_{k} + b\right)^{3}}, 
& \text{if } \lambda_k = -\dfrac{d}{c},
\end{cases}\label{(4.7)}
\end{equation}
\begin{equation}
\mathfrak{A}(\tilde{y}_{k+2}) =
\begin{cases}
\displaystyle  
\frac{\tilde{y}_{k + 2}(1)}{c\lambda_{k} + d} - \frac{c\tilde{y}_{k+1}(1)}{\left( c\lambda_{k} + d \right)^{2}}
+\frac{c^2y_{k}(1)}{\left( c\lambda_{k} + d \right)^{3}}, 
&  \text{if }  \lambda_k \neq -\dfrac{d}{c}, \\[10pt]
\displaystyle \ 
\frac{\tilde{y}'_{k + 2}(1)}{a\lambda_{k} + b} - \frac{a\tilde{y}'_{k+1}(1)}{\left( a\lambda_{k} + b \right)^{2}}
+\frac{a^2y'_{k}(1)}{\left( a\lambda_{k} + b\right)^{3}}, 
& \text{if } \lambda_k = -\dfrac{d}{c}.
\end{cases}\label{(4.8)}
\end{equation}
\begin{lem}
Suppose that \(\lambda_{k}\) is a triple eigenvalue and \(\lambda_{n}\) is a simple  eigenvalue \((\lambda_{k}=\lambda_{k+1}=\lambda_{k+2}\ne\lambda_{n})\).
Then 
\begin{equation}
  \left( y_{k + 2},y_{n} \right) = - (ad - bc)\mathfrak{A}(y_{k+2})\mathfrak{A}(y_n)
.\label{(4.9)}  
\end{equation}
\end{lem}
\begin{proof}
If $\lambda_k \neq -\dfrac{d}{c}\neq\lambda_n$, then we differentiate (\ref{(3.9)}) with respect to $\lambda$ twice and substitute  $\lambda=\lambda_k$ to obtain
\begin{equation*}
\left( 2\tilde{y}_{k + 2},y_{n} \right) = - (ad - bc)\left(\frac{2\tilde{y}_{k + 2}(1)}{c\lambda_{k} + d} - \frac{2c\tilde{y}_{k+1}(1)}{\left( c\lambda_{k} + d \right)^{2}}
+\frac{2c^2y_{k}(1)}{\left( c\lambda_{k} + d \right)^{3}}\right)\mathfrak{A}(y_{n}). 
\end{equation*}
We divide this by two and add (\ref{(3.10)}) multiplied by $C$ and (\ref{(3.11)}) multiplied by $D$.
If $\lambda_k \neq -\dfrac{d}{c}=\lambda_n$, then we differentiate (\ref{(3.12)}) with respect to $\mu$ and substitute $\mu=\lambda_k$ to obtain (\ref{(4.9)}) in this case.
Finally, if $\lambda_k = -\dfrac{d}{c}\neq\lambda_n$, then we need to consider two cases: $\lambda_n\neq-\frac{b}{a}$ and $\lambda_n=-\frac{b}{a}$. If  $\lambda_n\neq-\frac{b}{a}$, then 
we put in (\ref{(3.13)}) $\mu=\lambda_n$, differentiate twice with respect to $\lambda$ and substitute $\lambda=\lambda_k$ (noting $\omega(\lambda_k)=\omega^\prime(\lambda_k)=\omega^{\prime\prime}(\lambda_k)=0$) we obtain (\ref{(4.9)}) in this case.
If  $\lambda_n\neq-\frac{b}{a}$, then we differentiate (\ref{(3.14)}) with respect to $\lambda$ twice and then substitute $\lambda=\lambda_k$ to prove (\ref{(4.9)}) in this case.
\end{proof}
\begin{lem} Suppose that  $\lambda _{k}$ is an eigenvalue of multiplicity three \((\lambda_{k}=\lambda_{k+1} =\lambda_{k+2})\). Then
\begin{equation}
 (y_{k+2},y_{k})=- (ad - bc)\mathfrak{A}(y_{k+2})\cdot\mathfrak{A}(y_{k})+\mathfrak{A}(y_k)\frac{\omega^{\prime\prime\prime}(\lambda_k)}{6}. \label{(4.10)}
\end{equation}
\end{lem}
\begin{proof}
If $\lambda _{k}\ne -\frac{d}{c}$, then
we take the derivative of (\ref{(3.18)}) with respect to $\lambda$ and substitute $\lambda=\lambda_k$ to obtain
\begin{equation}
 (\tilde{y}_{k+2},y_{k})=- (ad - bc)\mathfrak{A}(\tilde{y}_{k+2})\cdot\mathfrak{A}(y_{k})+\mathfrak{A}(y_k)\frac{\omega^{\prime\prime\prime}(\lambda_k)}{6}. \label{(4.10.1)}
\end{equation}
By adding (\ref{(3.19)}) multiplied by $C$ and (\ref{(3.20)}) multiplied by $D$ to (\ref{(4.10.1)}) we obtain (\ref{(4.10)}).
If $\lambda _{k}=-\frac{d}{c}$, then we take the derivative of (\ref{(3.22)}) with respect to $\lambda$ and substitute $\lambda=\lambda_k$.
\end{proof}

\begin{lem} Suppose that  $\lambda _{k}$ is an eigenvalue of multiplicity three \((\lambda_{k}=\lambda_{k+1} =\lambda_{k+2})\). Then
\begin{equation}
 (y_{k+2},y_{k+1})=- (ad - bc)\mathfrak{A}(y_{k+2})\cdot\mathfrak{A}(y_{k+1})+\mathfrak{A}(\hat{y}_{k+1})\frac{\omega^{\prime\prime\prime}(\lambda_k)}{6}+ \mathfrak{A}({y}_{k})\frac{\omega^{IV}(\lambda_k)}{24}.\label{(4.11)}
\end{equation}
\end{lem}

\begin{proof}
If $\lambda _{k}\ne -\frac{d}{c}$, then
we take the derivative of (\ref{(3.25)}) with respect to $\mu$ twice and substitute $\mu=\lambda_k$ to obtain 
\begin{equation}
 (\tilde{y}_{k+2},\tilde{y}_{k+1})=- (ad - bc)\mathfrak{A}(\tilde{y}_{k+2})\cdot\mathfrak{A}(\tilde{y}_{k+1})+\mathfrak{A}(\tilde{y}_{k+1})\frac{\omega^{\prime\prime\prime}(\lambda_k)}{6}+ \mathfrak{A}({y}_{k})\frac{\omega^{IV}(\lambda_k)}{24}.\label{(4.12)}
\end{equation}
By adding (\ref{(3.26)}) multiplied by $C$ and (\ref{(3.19)}) multiplied by $D$ to (\ref{(4.12)}) we obtain 
\begin{equation}
 (y_{k+2},\tilde{y}_{k+1})=- (ad - bc)\mathfrak{A}(y_{k+2})\cdot\mathfrak{A}(\tilde{y}_{k+1})+\mathfrak{A}({y}_{k+1})\frac{\omega^{\prime\prime\prime}(\lambda_k)}{6}+ \mathfrak{A}({y}_{k})\frac{\omega^{IV}(\lambda_k)}{24}.\label{(4.13)}
\end{equation}
By adding (\ref{(4.10)}) multiplied by $C$ to (\ref{(4.13)}) we obtain (\ref{(4.11)}). If $\lambda _{k}=-\frac{d}{c}$, then we take the derivatives of (\ref{(2.9)}).
\end{proof}

\begin{lem} Suppose that  $\lambda _{k}$ is an eigenvalue of multiplicity three \((\lambda_{k}=\lambda_{k+1} =\lambda_{k+2})\). Then
$$
 {\left\| {y_{k+2}} \right\|}^{2}_{2}=(y_{k+2},y_{k+2})=- (ad - bc)(\mathfrak{A}(y_{k+2}))^2
 $$
 \begin{equation}
 +\mathfrak{A}(\hat{y}_{k+2})\frac{\omega^{\prime\prime\prime}(\lambda_k)}{6}+
 \mathfrak{A}(\hat{y}_{k+1})\frac{\omega^{IV}(\lambda_k)}{24}+ \mathfrak{A}({y}_{k})\frac{\omega^{V}(\lambda_k)}{120},\label{(4.14)}
\end{equation}
where $\hat{y}_{k+2}={y}_{k+2}+Cy_{k+1}+Dy_k$ and
\begin{equation}
\mathfrak{A}(\hat{y}_{k+2}) =
\begin{cases}
\displaystyle  
\frac{\hat{y}_{k + 2}(1)}{c\lambda_{k} + d} - \frac{c\hat{y}_{k+1}(1)}{\left( c\lambda_{k} + d \right)^{2}}
+\frac{c^2y_{k}(1)}{\left( c\lambda_{k} + d \right)^{3}}, 
&  \text{if }  \lambda_k \neq -\dfrac{d}{c}, \\[10pt]
\displaystyle \ 
\frac{\hat{y}'_{k + 2}(1)}{a\lambda_{k} + b} - \frac{a\hat{y}'_{k+1}(1)}{\left( a\lambda_{k} + b \right)^{2}}
+\frac{a^2y'_{k}(1)}{\left( a\lambda_{k} + b\right)^{3}}, 
& \text{if } \lambda_k = -\dfrac{d}{c}.
\end{cases}\label{(4.14.1)}
\end{equation}
\end{lem}

\begin{proof}
If $\lambda _{k}\ne -\frac{d}{c}$, then
we take the derivative of (\ref{(3.25)}) with respect to $\lambda$  twice and with respect to $\mu$ twice and substitute first $\lambda=\lambda_k$ and then $\mu=\lambda_k$ to obtain 
$$
 (\tilde{y}_{k+2},\tilde{y}_{k+2})=- (ad - bc)(\mathfrak{A}(\tilde{y}_{k+2}))^2
 $$
 \begin{equation}
 +\mathfrak{A}(\tilde{y}_{k+2})\frac{\omega^{\prime\prime\prime}(\lambda_k)}{6}+
 \mathfrak{A}(\tilde{y}_{k+1})\frac{\omega^{IV}(\lambda_k)}{24}+ \mathfrak{A}({y}_{k})\frac{\omega^{V}(\lambda_k)}{120}.\label{(4.15)}
\end{equation}
By adding (\ref{(4.12)}) multiplied by $C$ and (\ref{(4.10.1)}) multiplied by $D$ to (\ref{(4.15)}) we obtain 
$$
 (\tilde{y}_{k+2},{y}_{k+2})=- (ad - bc)\mathfrak{A}(\tilde{y}_{k+2})\cdot\mathfrak{A}({y}_{k+2})
 $$
\begin{equation}
+\mathfrak{A}({y}_{k+2})\frac{\omega^{\prime\prime\prime}(\lambda_k)}{6}+ \mathfrak{A}({y}_{k+1})\frac{\omega^{IV}(\lambda_k)}{24}+\mathfrak{A}({y}_{k})\frac{\omega^{V}(\lambda_k)}{120}.\label{(4.16)}
\end{equation}
By adding (\ref{(4.13)}) multiplied by $C$ and (\ref{(4.10)}) multiplied by $D$ to (\ref{(4.16)}) we obtain (\ref{(4.14)})
If $\lambda _{k}=-\frac{d}{c}$, then we take the derivatives of (\ref{(2.9)}).
\end{proof}
\section{Associated functions $y_{k+1}^*$, $y_{k+1}^\#$ and $y_{k+2}^\#$.} 
\begin{lem}
Let $\lambda_k$ be an eigenvalue of multiplicity two and
\(
y_{k+1}^* = y_{k+1} + C_1 y_k,
\)
where
\begin{equation}
    C_1=-\frac{\mathfrak{A}(\hat{y}_{k+1})\frac{\omega^{\prime\prime}(\lambda_k)}{2}+\mathfrak{A}(y_k)\frac{\omega^{\prime\prime\prime}(\lambda_k)}{6}}{\mathfrak{A}(y_k)\frac{\omega^{\prime\prime}(\lambda_k)}{2}}. \label{(5.1)}
\end{equation}
Then 
\begin{equation}
(y_{k+1}^*, y_{k+1}) =-(ad-bc)\mathfrak{A}(y_{k+1}^*)\mathfrak{A}(y_{k+1}), \label{(5.2)}
\end{equation}
where 
\begin{equation}
 \mathfrak{A}(y^{*}_{k+1}) =
\begin{cases}
\displaystyle  
\frac{y^{*}_{k + 1}(1)}{c\lambda_{k}+d} - \frac{cy_{k}(1)}{{(c\lambda_{k}+d)}^{2}}, 
& \text{if } \lambda_k \neq -\dfrac{d}{c}, \\[10pt]
\displaystyle \ 
\frac{(y^{*}_{k + 1})'(1)}{a\lambda_{k}+b} - \frac{ay'_{k}(1)}{{(a\lambda_{k}+b)}^{2}}, 
&  \text{if } \lambda_k = -\dfrac{d}{c}.
\end{cases} \label{(5.2.1)}
\end{equation}
\end{lem}

\begin{proof} By adding (\ref{(3.16)}) multiplied by $C_1$ to (\ref{(3.24)}) we obtain (\ref{(5.2)}).
\end{proof}
We note that the associated function $y_{k+1}^*$ satisfies (\ref{(3.8)}) and (\ref{(3.16)}), too.
From now on we will discuss only the case of a triple eigenvalue. It is worthwhile to note that if $\lambda_k$ is a triple eigenvalue, then $C_1$ is undefined.
\begin{lem}
Let $\lambda_k$ be an eigenvalue of multiplicity three and
\(
y_{k+1}^\# = y_{k+1} + C_2 y_k,
\)
where
\begin{equation}
    C_2=-\frac{\mathfrak{A}(\hat{y}_{k+1})\frac{\omega^{\prime\prime\prime}(\lambda_k)}{6}+\mathfrak{A}(y_k)\frac{\omega^{IV}(\lambda_k)}{24}}{\mathfrak{A}(y_k)\frac{\omega^{\prime\prime\prime}(\lambda_k)}{6}}. \label{(5.3)}
\end{equation}
Then 
\begin{equation}
(y_{k+1}^\#, y_{k+2}) =-(ad-bc)\mathfrak{A}(y_{k+1}^\#)\mathfrak{A}(y_{k+2}), \label{(5.4)}
\end{equation}
where
\begin{equation}
 \mathfrak{A}(y^{\#}_{k+1}) =
\begin{cases}
\displaystyle  
\frac{y^{\#}_{k + 1}(1)}{c\lambda_{k}+d} - \frac{cy_{k}(1)}{{(c\lambda_{k}+d)}^{2}}, 
& \text{if } \lambda_k \neq -\dfrac{d}{c}, \\[10pt]
\displaystyle \ 
\frac{(y^{\#}_{k + 1})'(1)}{a\lambda_{k}+b} - \frac{ay'_{k}(1)}{{(a\lambda_{k}+b)}^{2}}, 
& \text{if } \lambda_k = -\dfrac{d}{c}.
\end{cases} \label{(5.5.1)}  
\end{equation}

\end{lem}

\begin{proof} By adding (\ref{(4.10)}) multiplied by $C_2$ to (\ref{(4.11)}) we obtain (\ref{(5.4)}).
\end{proof}
We note that $y^{\#}_{k+1}$ satisfies (\ref{(3.8)}), (\ref{(3.16)}), and 
\begin{equation}
(y_{k+1}^\#, y_{k+1}) =-(ad-bc)\mathfrak{A}(y_{k+1}^\#)\mathfrak{A}(y_{k+1})+\mathfrak{A}(y_{k})\frac{\omega^{\prime\prime\prime}(\lambda_k)}{6} \label{(5.5)}
\end{equation}
Note that the function \(y^{*}_{k+2}=y_{k+2}+C_{2}y_{k+1}\), satisfies (\ref{(4.9)}) and (\ref{(4.10)}) with $\mathfrak{A}(y_{k+2})$ replaced by 
\begin{equation}
   \mathfrak{A}(y^{*}_{k+2}) =
\begin{cases}
\displaystyle  
\frac{y^{*}_{k + 2}(1)}{c\lambda_{k} + d} - \frac{cy^{\#}_{k+1}(1)}{\left( c\lambda_{k} + d \right)^{2}}
+\frac{c^2y_{k}(1)}{\left( c\lambda_{k} + d \right)^{3}}, 
& \text{if } \lambda_k \neq -\dfrac{d}{c}, \\[10pt]
\displaystyle \ 
\frac{(y^{*}_{k + 2})'(1)}{a\lambda_{k} + b} - \frac{a(y^{\#}_{k+1})'(1)}{\left( a\lambda_{k} + b \right)^{2}}
+\frac{a^2y'_{k}(1)}{\left( a\lambda_{k} + b\right)^{3}}, 
& \text{if } \lambda_k = -\dfrac{d}{c}.
\end{cases} \label{(5.5.2)}
\end{equation}
The function $y^{*}_{k + 2}$ also satisfies 
\begin{equation}
(y_{k+2}^*, y_{k+1}) =-(ad-bc)\mathfrak{A}(y_{k+2}^*)\mathfrak{A}(y_{k+1}), \label{(5.6)}
\end{equation}
\begin{equation}
(y_{k+2}^*, y_{k+2}) =-(ad-bc)\mathfrak{A}(y_{k+2}^*)\mathfrak{A}(y_{k+2})+J_k, \label{(5.7)}
\end{equation}
where 
 \begin{equation}
   J_k 
    ={\mathfrak{A}(\hat{y}_{k+2})\frac{\omega^{\prime\prime\prime}(\lambda_k)}{6}+\mathfrak{A}(\hat{y}_{k+1})\frac{\omega^{IV}(\lambda_k)}{24}}+{\mathfrak{A}(y_k)\frac{\omega^{V}(\lambda_k)}{120}}-C_2^2\mathfrak{A}(y_k)\frac{\omega^{\prime\prime\prime}(\lambda_k)}{6}\label{(5.8)}
\end{equation}

\begin{lem}
Let $\lambda_k$ be an eigenvalue of multiplicity three and
$y^{\#}_{k+2}=y^{*}_{k+2}+D_{1}y_{k}$
where
\begin{equation}
    D_1 =-\frac{J_k}{\mathfrak{A}(y_k)\frac{\omega^{\prime\prime\prime}(\lambda_k)}{6}}. \label{(5.9)}
\end{equation}
Then 
\begin{equation}
(y_{k+2}^\#, y_{k+2}) =-(ad-bc)\mathfrak{A}(y_{k+2}^\#)\mathfrak{A}(y_{k+2}), \label{(5.10)}
\end{equation}
where
\begin{equation}
\mathfrak{A}(y^{\#}_{k+2}) =
\begin{cases}
\displaystyle  
\frac{y^{\#}_{k + 2}(1)}{c\lambda_{k} + d} - \frac{cy^{\#}_{k+1}(1)}{\left( c\lambda_{k} + d \right)^{2}}
+\frac{c^2y_{k}(1)}{\left( c\lambda_{k} + d \right)^{3}}, 
& \text{if } \lambda_k \neq -\dfrac{d}{c}, \\[10pt]
\displaystyle \ 
\frac{(y^{\#}_{k + 2})'(1)}{a\lambda_{k} + b} - \frac{a(y^{\#}_{k+1})'(1)}{\left( a\lambda_{k} + b \right)^{2}}
+\frac{a^2y'_{k}(1)}{\left( a\lambda_{k} + b\right)^{3}}, 
& \text{if }  \lambda_k = -\dfrac{d}{c}.
\end{cases}\label{(5.5.3)}  
\end{equation}
\end{lem}

\begin{proof} By adding (\ref{(4.10)}) multiplied by $D_1$ to (\ref{(5.6)}) we obtain (\ref{(5.10)}).
\end{proof}
We note that $y^{\#}_{k+2}$ also satisfies (\ref{(4.9)}), (\ref{(4.10)}) with  $\mathfrak{A}(y_{k+2})$ replaced by $\mathfrak{A}(y_{k+2}^\#)$, and 
\begin{equation}
(y_{k+2}^\#, y_{k+1}) =-(ad-bc)\mathfrak{A}(y_{k+2}^\#)\mathfrak{A}(y_{k+1}). \label{(5.11)}
\end{equation}
\section{Minimality properties of the root functions}
In \cite{aliyev5} some results for the basis properties were proved.  Below we will provide a new form for the biorthogonal system using the characteristic function in these cases. Also, the cases with the necessary and sufficient conditions will be given a new form.
\subsection{Case (i).}

 In \cite{aliyev5} it was proved that in case (i), where  all the eigenvalues of (\ref{(1.1)})--(\ref{(1.3)}) are real and simple, the system
\begin{equation}
  \{ y_n \} \quad (n = 0,1,\ldots;\; n \neq l), \label{(6.1)} 
\end{equation}
where $l$ is a nonnegative integer, is minimal in space $L_2(0,1)$. In view of the results of the current paper the biorthogonal system
\begin{equation}
\{ u_n \} \quad (n = 0,1,\ldots;\; n \neq l), \label{(6.2)}
\end{equation}
satisfying $(u_n, y_m) = \delta_{nm}$, where $\delta_{nm}=0$ for $n \neq m$ and $\delta_{nn}=1$, can be written as 
\begin{equation}
u_n(x)=\frac{y_n(x)-\frac{\mathfrak{A}(y_{n})}{\mathfrak{A}(y_l)}\cdot y_l(x)}{ \mathfrak{A}(y_{n})\omega^\prime(\lambda_n)}
. \label{(6.3)}
\end{equation}
If $n\neq m$, then by Lemma 2.1,
\begin{equation}
(u_n, y_m)=\frac{(y_n,y_m)-\frac{\mathfrak{A}(y_{n})}{\mathfrak{A}(y_l)}\cdot (y_l, y_m)}{ \mathfrak{A}(y_{n})\omega^\prime(\lambda_n)}=0.
\label{(6.4)}
\end{equation}
If $n=m$, then by Lemma 2.1 and Lemma 2.3,
\begin{equation}
(u_n, y_n)=\frac{\|y_n\|_2^2-\frac{\mathfrak{A}(y_{n})}{\mathfrak{A}(y_l)}\cdot (y_l, y_n)}{ \mathfrak{A}(y_{n})\omega^\prime(\lambda_n)}=1.
\label{(6.5)}
\end{equation}

\subsection{Case (ii).}

In \cite{aliyev5} it was proved that in case (ii), where  all the eigenvalues of (\ref{(1.1)})--(\ref{(1.3)}) are real, and except one double eigenvalue $\lambda_k$, all are simple, the systems
\begin{equation}
  \{ y_n \} \quad (n = 0,1,\ldots;\; n \neq k+1), \label{(6.6)} 
\end{equation}
\begin{equation}
  \{ y_n \} \quad (n = 0,1,\ldots;\; n \neq l), \label{(6.12)} 
\end{equation}
where $l \ne k,\ k+1$ is a non-negative integer, are both minimal in space $L_2(0,1)$. For (\ref{(6.6)}) we define the biortogonal system for $n \ne k,\ k+1$ by
\begin{equation}
u_n(x)=\frac{y_n(x)-\frac{\mathfrak{A}(y_{n})}{\mathfrak{A}(y_k)}\cdot y_k(x)}{ \mathfrak{A}(y_{n})\omega^\prime(\lambda_n)}
, \label{(6.7)}
\end{equation}
and for $n = k$ by
\begin{equation}
u_k(x)=2\cdot\frac{y_{k+1}(x)-\frac{\mathfrak{A}(y_{k+1})}{\mathfrak{A}(y_k)}\cdot y_k(x)}{ \mathfrak{A}(y_k)\omega^{\prime\prime}(\lambda_k)}
. \label{(6.8)}
\end{equation}
The equality $(u_n, y_m) = \delta_{nm}$ for $ n,m\neq k+1$ can be checked using Lemma 2.1, Lemma 2.3, Corollary 2.4, Lemma 3.1 and Lemma 3.2.

For system (\ref{(6.8)}) the biorthogonal system  $n \ne k,\ k+1$ is defined by (\ref{(6.3)}) and for 
$n = k$ and $n = k+1$ by
\begin{equation}
u_k(x)=2\cdot\frac{y^*_{k+1}(x)-\frac{\mathfrak{A}(y^*_{k+1})}{\mathfrak{A}(y_l)}\cdot y_l(x)}{ \mathfrak{A}(y_k)\omega^{\prime\prime}(\lambda_k)}
, \label{(6.13)}
\end{equation}
\begin{equation}
u_{k+1}(x)=2\cdot\frac{y_k(x)-\frac{\mathfrak{A}(y_{k})}{\mathfrak{A}(y_l)}\cdot y_l(x)}{ \mathfrak{A}(y_{k})\omega^{\prime\prime}(\lambda_k)}
. \label{(6.14)}
\end{equation}

The following theorem gives an explicit form for the necessary and sufficient condition in \cite{aliyev5}.
\begin{thm} 
In case (ii), where  all the eigenvalues of (\ref{(1.1)})--(\ref{(1.3)}) are real, and except one double eigenvalue $\lambda_k$, all are simple, the system
\begin{equation}
  \{ y_n \} \quad (n = 0,1,\ldots;\; n \neq k), \label{(6.9)} 
\end{equation}
is minimal in space $L_2(0,1)$ if and only if $C\neq-\frac{\omega^{\prime\prime\prime}(\lambda_k)}{3\omega^{\prime\prime}(\lambda_k)}$.
\end{thm}

\begin{proof} For $n \ne k,\ k+1$ the biorthogonal system is consisted of
\begin{equation}
u_n(x)=\frac{y_n(x)-\frac{\mathfrak{A}(y_{n})}{\mathfrak{A}(y_{k+1}^{*})}\cdot y^*_{k+1}(x)}{ \mathfrak{A}(y_{n})\omega^\prime(\lambda_n)}
, \label{(6.10)}
\end{equation}
and for $n = k+1$,
\begin{equation}
u_{k+1}(x)=2\cdot\frac{y_{k}(x)-\frac{\mathfrak{A}(y_{k})}{\mathfrak{A}(y_{k+1}^{*})}\cdot y^*_{k+1}(x)}{ \mathfrak{A}(y_k)\omega^{\prime\prime}(\lambda_k)}
. \label{(6.11)}
\end{equation}
The equality $(u_n, y_m) = \delta_{nm}$ for $ n,m\neq k$ can be checked using again Lemma 2.1, Lemma 2.3, Lemma 3.1, Lemma 3.2, and Lemma 5.1.
If $\mathfrak{A}(y_{k+1}^{*}) = 0$, then by Lemma 3.1 and Lemma 5.1, $y_{k+1}^{*}$ is orthogonal to 
all the functions in (\ref{(6.9)}). Therefore, (\ref{(6.9)}) is not a complete system in space $L_{2}(0,1)$. Also, (\ref{(6.9)}) is not a minimal system in space $L_{2}(0,1)$, because otherwise as in \cite{aliyev4} we could use its minimality, combine with the quadratic closeness through asymptotic formula (\ref{(1.4)}), and prove that (\ref{(6.9)}) is a basis in space $L_{2}(0,1)$, which contradicts with the above fact that  (\ref{(6.9)}) is not a minimal system in space $L_{2}(0,1)$.
Note that condition $C\neq-\frac{\omega^{\prime\prime\prime}(\lambda_k)}{3\omega^{\prime\prime}(\lambda_k)}$ is equivalent to $\mathfrak{A}(y_{k+1}^{*})\ne 0$. Indeed, if $\lambda_k\neq-\frac{d}{c}$, then
\begin{equation}
 \mathfrak{A}(y^{*}_{k+1}) =
\frac{y^{*}_{k + 1}(1)}{c\lambda_{k}+d} - \frac{cy_{k}(1)}{{(c\lambda_{k}+d)}^{2}}\neq0, 
 \label{(6.11.1)}
\end{equation}
which is simplified using $y_{k+1}^* = y_{k+1} + C_1 y_k$, to
$$({c\lambda_{k}+d})({ y_{1}(k+1) + C_1 y_k(1)})\neq cy_k(1).$$
We use formula (\ref{(5.1)}) for $C_1$, \(
\hat{y}_{k+1} = y_{k+1} + Cy_k,
\) and obtain
$$
({c\lambda_{k}+d})\left( y_{k+1}(1) -\frac{\mathfrak{A}(\hat{y}_{k+1})\frac{\omega^{\prime\prime}(\lambda_k)}{2}+\mathfrak{A}(y_k)\frac{\omega^{\prime\prime\prime}(\lambda_k)}{6}}{\mathfrak{A}(y_k)\frac{\omega^{\prime\prime}(\lambda_k)}{2}}\cdot y_k(1)\right)\neq cy_k(1),$$
$$
({c\lambda_{k}+d})\left( y_{k+1}(1) -\frac{\frac{\hat{y}_{k+1}(1)}{c\lambda_{k}+d} - \frac{cy_{k}(1)}{{(c\lambda_{k}+d)}^{2}}+\frac{{y}_{k}(1)}{c\lambda_{k}+d}\frac{\omega^{\prime\prime\prime}(\lambda_k)}{3\omega^{\prime\prime}(\lambda_k)})}{\frac{{y}_{k}(1)}{c\lambda_{k}+d}}\cdot y_k(1)\right)\neq cy_k(1),$$
$$
({c\lambda_{k}+d})\left( y_{k+1}(1) -(y_{k+1}(1)+Cy_{k}(1)-\frac{cy_{k}(1)}{{(c\lambda_{k}+d)}}+y_{k}(1)\cdot\frac{\omega^{\prime\prime\prime}(\lambda_k)}{3\omega^{\prime\prime}(\lambda_k)}\right)\neq cy_k(1).
$$
By simplifying further we obtain
$C\neq-\frac{\omega^{\prime\prime\prime}(\lambda_k)}{3\omega^{\prime\prime}(\lambda_k)}$.
If $\lambda_k=-\frac{d}{c}$, then the proof is similar.
\end{proof}

\subsection{Case (iii).}

In \cite{aliyev5} it was proved that in case (iii), where  all the eigenvalues of (\ref{(1.1)})--(\ref{(1.3)}) are real, and except one triple eigenvalue $\lambda_k$, all are simple, the systems
\begin{equation}
  \{ y_n \} \quad (n = 0,1,\ldots;\; n \neq k+2), \label{(6.15)} 
\end{equation}
\begin{equation}
  \left\{ y_{n}\right\}  \ (n=0,\ 1,\ldots ; n \ne l),\label{(6.29)}  
\end{equation}
where $l \ne k,\ k+1,\ k+2$ is a non-negative integer,
are minimal in space $L_2(0,1)$.

For (\ref{(6.15)}) and $n \ne k,\ k+1,\ k+2$ the biorthogonal system is defined by (\ref{(6.7)}) and for 
$n = k$ and $n = k+1$ by
\begin{equation}
u_k(x)=6\cdot\frac{y^{\#}_{k+2}(x)-\frac{\mathfrak{A}(y^{\#}_{k+2})}{\mathfrak{A}(y_k)}\cdot y_k(x)}{ \mathfrak{A}(y_k)\omega^{\prime\prime\prime}(\lambda_k)}
, \label{(6.16)}
\end{equation}
\begin{equation}
u_{k+1}(x)=6\cdot\frac{y_{k+1}(x)-\frac{\mathfrak{A}(y_{k+1})}{\mathfrak{A}(y_k)}\cdot y_k(x)}{ \mathfrak{A}(y_{k})\omega^{\prime\prime\prime}(\lambda_k)}
, \label{(6.17)}
\end{equation}
respectively. The equality $(u_n, y_m) = \delta_{nm}$ for $ n,m\neq k+2$ can be verified again using Lemma 2.1, Lemma 2.3, Lemma 3.1, Lemma 3.2, Lemma 3.3 and Lemma 5.3.

For (\ref{(6.29)}) and $n \ne k,\ k+1,\ k+2$ the biorthogonal system is defined by (\ref{(6.3)}) and for 
$n = k$, $n = k+1$, and $n = k+2$ by
\begin{equation}
u_k(x)=6\cdot\frac{y^{\#}_{k+2}(x)-\frac{\mathfrak{A}(y^{\#}_{k+2})}{\mathfrak{A}(y_l)}\cdot y_l(x)}{ \mathfrak{A}(y_k)\omega^{\prime\prime\prime}(\lambda_k)}
, \label{(6.30)}
\end{equation}
\begin{equation}
u_{k+1}(x)=6\cdot\frac{y^{\#}_{k+1}(x)-\frac{\mathfrak{A}(y^{\#}_{k+1})}{\mathfrak{A}(y_l)}\cdot y_l(x)}{ \mathfrak{A}(y_k)\omega^{\prime\prime\prime}(\lambda_k)}
, \label{(6.31)}
\end{equation}
\begin{equation}
u_{k+2}(x)=6\cdot\frac{y_{k}(x)-\frac{\mathfrak{A}(y_{k})}{\mathfrak{A}(y_l)}\cdot y_l(x)}{ \mathfrak{A}(y_k)\omega^{\prime\prime\prime}(\lambda_k)}
, \label{(6.32)}
\end{equation}
respectively.

\begin{thm} 
In case (iii), where  all the eigenvalues of (\ref{(1.1)})--(\ref{(1.3)}) are real, and except one triple eigenvalue $\lambda_k$, all are simple, the system
\begin{equation}
  \{ y_n \} \quad (n = 0,1,\ldots;\; n \neq k+1), \label{(6.19)} 
\end{equation}
is minimal in space $L_2(0,1)$ if and only if $C\neq-\frac{\omega^{IV}(\lambda_k)}{4\omega^{\prime\prime\prime}(\lambda_k)}$.
\end{thm}

\begin{proof}
 For $n \ne k,\ k+1,\ k+2$ the biorthogonal system is consisted of
\begin{equation}
u_n(x)=\frac{y_n(x)-\frac{\mathfrak{A}(y_{n})}{\mathfrak{A}(y_{k+1}^{\#})}\cdot y^{\#}_{k+1}(x)}{ \mathfrak{A}(y_{n})\omega^\prime(\lambda_n)}
, \label{(6.20)}
\end{equation}
and for $n = k$ and $n = k+2$,
\begin{equation}
u_k(x)=6\cdot\frac{y^{\#}_{k+2}(x)-\frac{\mathfrak{A}(y^{\#}_{k+2})}{\mathfrak{A}(y^{\#}_{k+1})}\cdot y^{\#}_{k+1}(x)}{ \mathfrak{A}(y_k)\omega^{\prime\prime\prime}(\lambda_k)}
, \label{(6.21)}
\end{equation}
\begin{equation}
u_{k+2}(x)=6\cdot\frac{y_{k}(x)-\frac{\mathfrak{A}(y_{k})}{\mathfrak{A}(y^{\#}_{k+1})}\cdot y^{\#}_{k+1}(x)}{ \mathfrak{A}(y_{k})\omega^{\prime\prime\prime}(\lambda_k)}
, \label{(6.22)}
\end{equation}
respectively.
The equality $(u_n, y_m) = \delta_{nm}$ for $ n,m\neq k+2$ can be verified again using Lemma 5.2 and Lemma 5.3.
If
\(
\mathfrak{A}\bigl(y_{k+1}^{\#}\bigr)=0,
\)
then $y_{k+1}^{\#}(x)$ is orthogonal to 
all the functions in (\ref{(6.19)}). Therefore, (\ref{(6.19)}) is neither complete nor minimal in $L_2(0,1)$.
Again, note that condition $C\neq-\frac{\omega^{IV}(\lambda_k)}{4\omega^{\prime\prime\prime}(\lambda_k)}$ is equivalent to $\mathfrak{A}(y_{k+1}^{\#})\ne 0$. If $\lambda_k\neq-\frac{d}{c}$, then
\begin{equation}
 \mathfrak{A}(y^{\#}_{k+1}) =
\frac{y^{\#}_{k + 1}(1)}{c\lambda_{k}+d} - \frac{cy_{k}(1)}{{(c\lambda_{k}+d)}^{2}}\neq0
 \label{(6.20.1)}
\end{equation}
is simplified using $y_{k+1}^\# = y_{k+1} + C_2 y_k$, $\hat{y}_{k+1}={y}_{k+1}+Cy_k$,
$\mathfrak{A}(\hat{y}_{k+1}) =
\frac{\hat{y}_{k + 1}(1)}{c\lambda_{k}+d} - \frac{cy_{k}(1)}{{(c\lambda_{k}+d)}^{2}}$, 
and (\ref{(5.3)}), to obtain
$$
\frac{{y}_{k+1}(1)}{y_{k}(1)}-\left(\frac{\mathfrak{A}(\hat{y}_{k+1})}{\mathfrak{A}({y}_{k})}+\frac{\omega^{IV}(\lambda_k)}{4\omega^{\prime\prime\prime}(\lambda_k)}\right)\neq \frac{c}{c\lambda_{k}+d}, $$

$$
\frac{{y}_{k+1}(1)}{y_{0}(1)}- \left(\frac{\frac{\hat{y}_{k+1}(1)}{c\lambda_{k}+d} - \frac{cy_{k}(1)}{{(c\lambda_{k}+d)}^{2}}}{\frac{y_{k}(1)}{{(c\lambda_{k}+d)}}}+\frac{\omega^{IV}(\lambda_k)}{4\omega^{\prime\prime\prime}(\lambda_k)}\right)\neq\frac{c}{c\lambda_{k}+d} ,
$$
$$
\frac{{y}_{k+1}(1)}{y_{k}(1)}-\frac{{y}_{k+1}(1)}{y_{k}(1)}+C+ \frac{c}{c\lambda_{k}+d}+\frac{\omega^{IV}(\lambda_k)}{4\omega^{\prime\prime\prime}(\lambda_k)}\neq\frac{c}{c\lambda_{k}+d}.
$$
By simplifying further, we obtain
$C\neq-\frac{\omega^{IV}(\lambda_k)}{4\omega^{\prime\prime\prime}(\lambda_k)}$.
The case $\lambda_k=-\frac{d}{c}$ is considered analogously.
\end{proof}
\begin{thm}
In case (iii), where  all the eigenvalues of (\ref{(1.1)})--(\ref{(1.3)}) are real, and except one triple eigenvalue $\lambda_k$, all are simple, the system
\begin{equation}
  \{ y_n \} \quad (n = 0,1,\ldots;\; n \neq k), \label{(6.24)} 
\end{equation}
is minimal in space $L_2(0,1)$ if and only if
\begin{equation}
 D\neq C^2+
\frac{\omega^{IV}(\lambda_k)}{4\omega^{\prime\prime\prime}(\lambda_k)}\left(C+\frac{\omega^{IV}(\lambda_k)}{4\omega^{\prime\prime\prime}(\lambda_k)}\right)-\frac{\omega^{V}(\lambda_k)}{20\omega^{\prime\prime\prime}(\lambda_k)}. \label{(6.25)}  
\end{equation}
\end{thm}

\begin{proof} For $n \ne k,\ k+1,\ k+2$ the biorthogonal system is consisted of
\begin{equation}
u_n(x)=\frac{y_n(x)-\frac{\mathfrak{A}(y_{n})}{\mathfrak{A}(y_{k+2}^{\#})}\cdot y^{\#}_{k+2}(x)}{ \mathfrak{A}(y_{n})\omega^\prime(\lambda_n)}
, \label{(6.26)}
\end{equation}
and for $n = k+1$ and $n = k+2$,
\begin{equation}
u_{k+1}(x)=6\cdot\frac{y_{k+1}(x)-\frac{\mathfrak{A}(y_{k+1})}{\mathfrak{A}(y^{\#}_{k+2})}\cdot y^{\#}_{k+2}(x)}{ \mathfrak{A}(y_{k})\omega^{\prime\prime\prime}(\lambda_k)}
, \label{(6.27)}
\end{equation}
\begin{equation}
u_{k+2}(x)=6\cdot\frac{y_{k}(x)-\frac{\mathfrak{A}(y_{k})}{\mathfrak{A}(y^{\#}_{k+2})}\cdot y^{\#}_{k+2}(x)}{ \mathfrak{A}(y_{k})\omega^{\prime\prime\prime}(\lambda_k)}
, \label{(6.28)}
\end{equation}
respectively. One can check that $\mathfrak{A}(y_{k+2}^{\#})\neq 0$ is equivalent to (\ref{(6.25)}).
Indeed, if $\lambda_k\neq-\frac{d}{c}$, then
\begin{equation}
 \mathfrak{A}(y^{\#}_{k+2}) =
\frac{y^{\#}_{k + 2}(1)}{c\lambda_{k} + d} - \frac{cy^{\#}_{k+1}(1)}{\left( c\lambda_{k} + d \right)^{2}}
+\frac{c^2y_{k}(1)}{\left( c\lambda_{k} + d \right)^{3}}\neq0
 \label{(6.28.1.1)}
\end{equation}
is simplified using 
\(
y_{k+1}^\# = y_{k+1} + C_2 y_k,
\)
$\hat{y}_{k+1}={y}_{k+1}+Cy_k$,
$\hat{y}_{k+2}={y}_{k+2}+Cy_{k+1}+Dy_k$,
\(y^{*}_{k+2}=y_{k+2}+C_{2}y_{k+1}\),
$y^{\#}_{k+2}=y^{*}_{k+2}+D_{1}y_{k}$,
(\ref{(5.3)}), (\ref{(5.8)}), and (\ref{(5.9)}), to
$$
\frac{y_{k+2}(1) + C_2 y_{k+1}(1) + D_1 y_k(1)}{y_k(1)} - c\frac{y_{k+1}(1) + C_2 y_k(1)}{(c\lambda_k+d)y_k(1)} + \frac{c^2}{(c\lambda_k+d)^2} \neq 0,
$$

$$
\frac{y_{k+2}(1)}{y_k(1)} + C_2\frac{y_{k+1}(1)}{y_k(1)} - \left( \frac{\mathfrak{A}(\hat{y}_{k+2})}{\mathfrak{A}(y_k)} + \frac{\mathfrak{A}(\hat{y}_{k+1})}{\mathfrak{A}(y_k)}\frac{\omega^{IV}(\lambda_k)}{4\omega^{\prime\prime\prime}(\lambda_k)} + \frac{\omega^{V}(\lambda_k)}{20\omega^{\prime\prime\prime}(\lambda_k)} - C_2^2 \right)$$
$$- c\frac{y_{k+1}(1) + C_2 y_k(1)}{(c\lambda_k+d)y_k(1)} + \frac{c^2}{(c\lambda_k+d)^2} \neq 0,
$$

$$
\frac{y_{k+2}(1)}{y_k(1)} + C_2\left(\frac{y_{k+1}(1)}{y_k(1)} - \frac{c}{c\lambda_k+d}\right) + C_2^2 $$
$$- \left( \frac{\frac{\hat{y}_{k+2}(1)}{c\lambda_k+d} - \frac{c\hat{y}_{k+1}(1)}{(c\lambda_k+d)^2} + \frac{c^2y_k(1)}{(c\lambda_k+d)^3}}{\frac{y_k(1)}{c\lambda_k+d}} + \frac{\mathfrak{A}(\hat{y}_{k+1})}{\mathfrak{A}(y_k)}\frac{\omega^{IV}(\lambda_k)}{4\omega^{\prime\prime\prime}(\lambda_k)} + \frac{\omega^{V}(\lambda_k)}{20\omega^{\prime\prime\prime}(\lambda_k)} \right) $$
$$- c\frac{y_{k+1}(1)}{(c\lambda_k+d)y_k(1)} + \frac{c^2}{(c\lambda_k+d)^2} \neq 0,
$$
$$
C_2\left(\frac{y_{k+1}(1)}{y_k(1)} - \frac{c}{c\lambda_k+d}\right) + C_2^2 - C\frac{y_{k+1}(1)}{y_k(1)} - D + \frac{cC}{c\lambda_k+d} - \frac{\mathfrak{A}(\hat{y}_{k+1})}{\mathfrak{A}(y_k)}\frac{\omega^{IV}(\lambda_k)}{4\omega^{\prime\prime\prime}(\lambda_k)} - \frac{\omega^{V}(\lambda_k)}{20\omega^{\prime\prime\prime}(\lambda_k)} \neq 0.
$$
Substituting $C_2 = -\frac{\mathfrak{A}(\hat{y}_{k+1})}{\mathfrak{A}(y_k)} - \frac{\omega^{IV}(\lambda_k)}{4\omega^{\prime\prime\prime}(\lambda_k)}$ and utilizing the relation $\frac{\mathfrak{A}(\hat{y}_{k+1})}{\mathfrak{A}(y_k)} = \frac{y_{k+1}(1)}{y_k(1)} + C - \frac{c}{c\lambda_k+d}$, we obtain:

$$
-\left(\frac{y_{k+1}(1)}{y_k(1)} + C - \frac{c}{c\lambda_k+d} + \frac{\omega^{IV}(\lambda_k)}{4\omega^{\prime\prime\prime}(\lambda_k)}\right)\left(\frac{y_{k+1}(1)}{y_k(1)} - \frac{c}{c\lambda_k+d}\right)$$
$$+ \left(\frac{y_{k+1}(1)}{y_k(1)} + C - \frac{c}{c\lambda_k+d} + \frac{\omega^{IV}(\lambda_k)}{4\omega^{\prime\prime\prime}(\lambda_k)}\right)^2 
$$
$$
- C\frac{y_{k+1}(1)}{y_k(1)} - D + \frac{cC}{c\lambda_k+d} - \left(\frac{y_{k+1}(1)}{y_k(1)} + C - \frac{c}{c\lambda_k+d}\right)\frac{\omega^{IV}(\lambda_k)}{4\omega^{\prime\prime\prime}(\lambda_k)} - \frac{\omega^{V}(\lambda_k)}{20\omega^{\prime\prime\prime}(\lambda_k)} \neq 0,
$$
$$
\left(\frac{y_{k+1}(1)}{y_k(1)} - \frac{c}{c\lambda_k+d}\right)\left(C + \frac{\omega^{IV}(\lambda_k)}{4\omega^{\prime\prime\prime}(\lambda_k)}\right) + C\left(C + \frac{\omega^{IV}(\lambda_k)}{4\omega^{\prime\prime\prime}(\lambda_k)}\right) + \frac{\omega^{IV}(\lambda_k)}{4\omega^{\prime\prime\prime}(\lambda_k)}\left(C + \frac{\omega^{IV}(\lambda_k)}{4\omega^{\prime\prime\prime}(\lambda_k)}\right) 
$$
$$
- C\left(\frac{y_{k+1}(1)}{y_k(1)} - \frac{c}{c\lambda_k+d}\right) - D - \left(\frac{y_{k+1}(1)}{y_k(1)} - \frac{c}{c\lambda_k+d}\right)\frac{\omega^{IV}(\lambda_k)}{4\omega^{\prime\prime\prime}(\lambda_k)} - C\frac{\omega^{IV}(\lambda_k)}{4\omega^{\prime\prime\prime}(\lambda_k)} - \frac{\omega^{V}(\lambda_k)}{20\omega^{\prime\prime\prime}(\lambda_k)} \neq 0.
$$
By simplifying further, we obtain:
$$
C^2 + \frac{\omega^{IV}(\lambda_k)}{4\omega^{\prime\prime\prime}(\lambda_k)}\left(C + \frac{\omega^{IV}(\lambda_k)}{4\omega^{\prime\prime\prime}(\lambda_k)}\right) - \frac{\omega^{V}(\lambda_k)}{20\omega^{\prime\prime\prime}(\lambda_k)} - D \neq 0,
$$
which directly yields to (\ref{(6.25)}).
The case $\lambda_k=-\frac{d}{c}$ is considered analogously. If $\mathfrak{A}(y_{k+2}^{\#})=0$, then the system (\ref{(6.24)}) is not minimal in $L_2(0,1)$.
\end{proof}

\subsection{Case (iv).}

In \cite{aliyev5} it was proved that in case (iv), where  all the eigenvalues of (\ref{(1.1)})--(\ref{(1.3)}) are simple, and except one pair of
complex conjugate non-real eigenvalues $\lambda_r$ and $\lambda_s=\overline{\lambda_r}$, all are real, the systems
\begin{equation}
 \left\{ y_n \right\} \quad (n=0,1,\ldots;\ n\neq r),\label{(6.33)}
\end{equation}
\begin{equation}
\left\{ y_n \right\} \quad (n=0,1,\ldots;\ n\neq l), \label{(6.34)}
\end{equation}
where $l$ is a non-negative integer and $l\neq r,s$, are minimal in space $L_2(0,1)$.

 For (\ref{(6.33)}) and $n \ne r,s$ the biorthogonal system is defined by
\begin{equation}
u_n(x)=\frac{y_n(x)-\frac{\mathfrak{A}(y_{n})}{\mathfrak{A}(y_s)}\cdot y_s(x)}{ \mathfrak{A}(y_{n})\omega^\prime(\lambda_n)}
,\label{(6.35)}
\end{equation}
and for $n =s$ by
\begin{equation}
u_s(x)=\frac{y_r(x)-\frac{\mathfrak{A}(y_{r})}{\mathfrak{A}(y_s)}\cdot y_s(x)}{ \mathfrak{A}(y_{s})\omega^\prime(\lambda_s)}
. \label{(6.36)}
\end{equation}
For (\ref{(6.34)}) and $n \ne r,s,l$ the biorthogonal system is defined by (\ref{(6.3)}),
and for $n =r$ and $n =s$ by
\begin{equation}
u_r(x)=\frac{y_s(x)-\frac{\mathfrak{A}(y_{s})}{\mathfrak{A}(y_l)}\cdot y_l(x)}{ \mathfrak{A}(y_{r})\omega^\prime(\lambda_r)}
, \label{(6.37)}
\end{equation}
\begin{equation}
u_s(x)=\frac{y_r(x)-\frac{\mathfrak{A}(y_{r})}{\mathfrak{A}(y_l)}\cdot y_l(x)}{ \mathfrak{A}(y_{s})\omega^\prime(\lambda_s)}
, \label{(6.38)}
\end{equation}
respectively.

\section{Examples.} As a demonstration of the theory, we will present several worked out examples from \cite{aliyev5} using the methods of the current paper. In all of the examples the results below and the results from \cite{aliyev5} coincide.

\subsection*{Example 1.}Let us take the problem  
$$
-y^{\prime \prime }=\lambda y,\ 0<x<1,
$$
$$
y'(0)=0,\ 
{\lambda}y(1)=\left(\frac{4}{\pi^2}{\lambda }-1\right)y^{\prime }(1).
$$
The solution of problem $-y^{\prime \prime }=\lambda y$ with boundary conditions $y(0)=1$, $y^\prime(0)=0$ is $y({x,\lambda)}=\cos{\sqrt{\lambda}x}$. Therefore, $y^\prime({x,\lambda)}=-{\sqrt{\lambda}}\sin{\sqrt{\lambda}x}$, Consequently, the characteristic function is$$\omega(\lambda) = {\lambda}\cos{\sqrt{\lambda}}+\left(\frac{4}{\pi^2}{\lambda }-1\right){\sqrt{\lambda}}\sin{\sqrt{\lambda}}.
$$
The Taylor expansion of the characteristic function is \[\omega(\lambda) =  \left(\frac{4}{\pi^2} - \frac{1}{3}\right)\lambda^2 + \left(\frac{1}{30} - \frac{2}{3\pi^2}\right)\lambda^3 + O(\lambda^4).\]
So, $\omega(0)=\omega^\prime(0)=0$, $\omega^{\prime\prime}(0)=\frac{-2 \pi^{2}+24}{3 \pi^{2}}$, and $\omega^{\prime\prime\prime}(0)=\frac{\pi^{2}-20}{5 \pi^{2}}$.
Then
$\lambda _{0}=\lambda _{1}=0$ is a double eigenvalue, and $\lambda _{0}\neq-\frac{d}{c}$. Similarly,
$$\omega(\lambda) = \left(\frac{2}{\pi} - \frac{\pi}{4}\right)\left(\lambda - \frac{\pi^2}{4}\right) +  O\left(\left(\lambda - \frac{\pi^2}{4}\right)^2\right).$$
Then the second eigenvalue $\lambda _{2}=\frac{\pi^2}{4}$ is simple and $\lambda _{2}=-\frac{d}{c}$. The remaining eigenvalues $\lambda _{3}<\lambda _{4}<\ldots$ are the other positive zeros  of the characteristic function.

The eigenfunctions are $y_{0}=1$, $y_{2}=\cos{\frac{\pi x}{2}}$, $y_{n}=\cos{\sqrt{\lambda _{n}}x}$ $(n\ge 3)$, and the first associated function of $y_{0}$ is defined as $y_{1}=\lim_{\lambda\to 0}y_\lambda(x,\lambda)+Cy_0=-\frac{1}{2}x^2+C$, where $C$ is
an arbitrary constant.
By Theorem 6.1,
the system $\left\{ y_{1},\ y_{2},\ y_{3},\ldots\right\}$,  with removed $y_{0}$ is minimal in $L_2(0,1)$ if and only if $C\ne-\frac{\omega^{\prime\prime\prime}(\lambda_0)}{3\omega^{\prime\prime}(\lambda_0)}=\frac{\mathrm{\pi}^{2}-20}{10 \left(\mathrm{\pi}^{2}-12\right)}$.

\subsection*{Example 2.}Consider the problem  
$$
-y^{\prime \prime }=\lambda y,\ 0<x<1,
$$
$$
y'(0)=0,\ 
3{\lambda}y(1)=\left( {\lambda }-3\right)y^{\prime }(1).
$$
The solution of problem $-y^{\prime \prime }=\lambda y$ with boundary conditions $y(0)=1$, $y^\prime(0)=0$ is $y({x,\lambda)}=\cos{\sqrt{\lambda}x}$. Therefore, $y^\prime({x,\lambda)}=-\sqrt{\lambda}\sin{\sqrt{\lambda}x}$, and
the characteristic function is \[\omega(\lambda) = 3\lambda\cos{\sqrt{\lambda}}+  (\lambda - 3)\sqrt{\lambda}\sin{\sqrt{\lambda}}=-\frac{\lambda^{3}}{15}+\frac{\lambda^{4}}{210}-\frac{\lambda^{5}}{7560}+O(\lambda^{6}).\]
Consequently, $\omega(0)=\omega^\prime(0)=\omega^{\prime\prime}(0)=0$, $\omega^{\prime\prime\prime}(0)=-\frac{2}{5}$, $\omega^{IV}(0)=\frac{4}{35}$, and $\omega^{V}(0)=-\frac{1}{63}$.
Finally, $\lambda _{0}=\lambda _{1}=\lambda _{2}=0$ is the triple eigenvalue such that $\lambda _{0}\neq-\frac{d}{c}=3$. All other eigenvalues $\lambda _{3}<\lambda _{4}<\ldots$ are real and simple.

The eigenfunctions are $y_{0}=1$ and $y_{n}=\cos{\sqrt{\lambda _{n}}x}$ $(n\ge 3)$.
To find the first associated function of $y_{0}$ we need to calculate the limit
$$\tilde{y}_1=\lim_{\lambda\to 0}y_\lambda(x,\lambda)=-\frac{x^2}{2}.$$
Similarly, for the second associated function
$$\tilde{y}_2=\frac{1}{2}\lim_{\lambda\to 0}y_{\lambda\lambda}(x,\lambda)=\frac{x^4}{24}.$$
Thus, the first and second associated functions are $y_{1}=-\frac{1}{2}x^2+C$ and $y_{2}=\frac{1}{24}x^4-\frac{C}{2}x^2+D$, where $C$ and $D$ are constants. By Theorem 6.2,
the system $\left\{ y_{0},\ y_{2},\ y_{3},\ldots\right\}$, is minimal in $L_2(0,1)$ if and only if $C\neq-\frac{\omega^{IV}(\lambda_k)}{4\omega^{\prime\prime\prime}(\lambda_k)}=\frac{1}{14}$. 
Similarly, according to Theorem 6.3,
the system $\left\{ y_{1},\ y_{2},\ y_{3},\ldots\right\}$ is minimal in $L_2(0,1)$ if and only if (\ref{(6.25)}) is satisfied. By simplifying, we obtain that (\ref{(6.25)}) is equivalent to $D\ne\frac{11}{3528}-\frac{1}{14} C+C^{2}$.

\subsection*{Example 3.}Consider the problem  
$$
-y^{\prime \prime }=\lambda y,\ 0<x<1,
$$
$$
y(0)=-y'(0),\ 
(9{\lambda}+15)y(1)= 5{\lambda }y^{\prime }(1).
$$
The solution of problem $-y^{\prime \prime }=\lambda y$ with boundary conditions $y(0)=1$, $y^\prime(0)=-1$ is $y({x,\lambda)}=\cos{\sqrt{\lambda}x}-\frac{1}{\sqrt{\lambda}}\sin{\sqrt{\lambda}x}$. Therefore, $$y^\prime({x,\lambda)}=-\sqrt{\lambda}\sin{\sqrt{\lambda}x}-\cos{\sqrt{\lambda}x},$$and
the characteristic function is $$\omega(\lambda) = (9{\lambda}+15)\left(\cos{\sqrt{\lambda}}-\frac{\sin{\sqrt{\lambda}}}{{\sqrt{\lambda}}} \right)+5\lambda\left(\sqrt{\lambda}\sin{\sqrt{\lambda}} +\cos{\sqrt{\lambda}}\right)
$$
$$
=-\frac{12\lambda^{3}}{35}+\frac{23\lambda^{4}}{945}-\frac{\lambda^{5}}{1485}+O(\lambda^{6}).
$$
Consequently, $\omega(0)=\omega^\prime(0)=\omega^{\prime\prime}(0)=0$, $\omega^{\prime\prime\prime}(0)=-\frac{72}{35}$, $\omega^{IV}(0)=-\frac{1}{1485}$, and $\omega^{V}(0)=-\frac{8}{99}$.
Finally, $\lambda _{0}=\lambda _{1}=\lambda _{2}=0$ is the triple eigenvalue such that $\lambda _{0}=-\frac{d}{c}=0$. All other eigenvalues $\lambda _{3}<\lambda _{4}<\ldots$ are real and simple.
The eigenfunctions are $y_{0}=1-x$ and $y_{n}=\cos{\sqrt{\lambda _{n}}x}-\frac{\sin{\sqrt{\lambda_n}x}}{{\sqrt{\lambda_n}}}$ $(n\ge 3)$.
To find the first associated function of $y_{0}$ we need to calculate the limit
$$\tilde{y}_1=\lim_{\lambda\to 0}y_\lambda(x,\lambda)=\frac{x^3}{6}-\frac{x^2}{2}.$$
Similarly, for the second associated function
$$\tilde{y}_2=\frac{1}{2}\lim_{\lambda\to 0}y_{\lambda\lambda}(x,\lambda)=-\frac{x^5}{120}+\frac{x^4}{24}.$$
Thus, the first and second associated functions are 
$y_{1}=\frac{x^{3}}{6}-\frac{x^{2}}{2}+C\cdot \left(1-x\right)$, and $y_{2}=-\frac{x^{5}}{120}+\frac{x^{4}}{24}+C\cdot \left(\frac{x^{3}}{6}-\frac{x^{2}}{2}\right)+{D}\cdot \left(1-x\right)$, where $C$ and $D$ are constants. 
By Theorem 6.2,
the system $\left\{ y_{0},\ y_{2},\ y_{3},\ldots\right\}$, is minimal in $L_2(0,1)$ if and only if $C\neq-\frac{\omega^{IV}(\lambda_k)}{4\omega^{\prime\prime\prime}(\lambda_k)}=\frac{23}{324}$. 
Similarly, by Theorem 6.3,
the system $\left\{ y_{1},\ y_{2},\ y_{3},\ldots\right\}$ is minimal in $L_2(0,1)$ if and only if (\ref{(6.25)}) is satisfied. By simplifying we obtain that (\ref{(6.25)}) is equivalent to $D\ne\frac{3551}{1154736}-\frac{23}{3244} C+C^{2}$.

\subsection*{Acknowledgment}
This work was financially supported by ADA University and Baku State University.

\section*{Declarations}

\noindent\textbf{Ethics approval and consent to participate.}
Not applicable.

\noindent\textbf{Ethics declaration} Not applicable.

\noindent\textbf{Ethics, Consent to Participate, and Consent to Publish declarations.} Not applicable.

\noindent\textbf{Consent for publication.}
Not applicable.

\noindent\textbf{Competing interests.}
The authors declare that they have no competing interests.

\noindent\textbf{Authors' contributions.}
Both authors contributed equally to the conception, analysis, and writing of the manuscript. Both authors read and approved the final manuscript.

\noindent\textbf{Funding.}
This work was supported by the ADA University Faculty Research and Development Fund and Baku State University.

\noindent\textbf{Availability of data and materials.}
No datasets were generated or analysed during the current study. Data sharing is not applicable to this article.

\noindent\textbf{Acknowledgements.}
Not applicable.



\begin{thebibliography}{1}
\bibitem{aliyev0}
\textbf{N. B. Kerimov, Y. N. Aliyev}, \textit{The basis property in $L_p$ of the boundary value problem rationally dependent on the eigenparameter}, Studia Mathematica \textbf{174(2)} (2006), 201--212.

\bibitem{aliyev1}
\textbf{Y. N. Aliyev}, \textit{Minimality of the system of root functions of Sturm-Liouville problems with decreasing affine boundary conditions}, Colloquium Mathematicum \textbf{109(1)}
(2007), 147-162.

\bibitem{aliyev2}
\textbf{Y.N. Aliyev,} Minimality Properties of Sturm-Liouville Problems with Increasing Affine Boundary Conditions. In: Bastos M.A., Castro L., Karlovich A.Y. (eds) Operator Theory, Functional Analysis and Applications. Operator Theory: Advances and Applications, Birkhäuser, Cham.  vol 282 (2021) 33-49.

\bibitem{aliyev3}
\textbf{Y. N. Aliyev, N. B. Kerimov}, The Basis Property Of Sturm–Liouville Problems With Boundary Conditions Depending Quadratically On The Eigenparameter, The Arabian Journal For Science And Engineering  \textbf{33(1)} (2008),123-136.

\bibitem{aliyev4}
\textbf{Y. Aliyev,} Necessary and Sufficient Conditions for Basis Properties of the System of Root Functions of Sturm-Liouville Boundary Value Problems with Eigenparameter Dependent Boundary Conditions. In: Ashyralyev, A., Ruzhansky, M., Sadybekov, M.A. (eds)
 Analysis and Applied Mathematics. AAM 2022. Trends in Mathematics, vol 6. Birkhäuser, Cham.  (2024) 21-32.

 \bibitem{aliyev5}
\textbf{Y. Aliyev, N. Aliyeva} Sturm-Liouville problems with a boundary condition depending linearly on an eigenparameter. Preprint (2026) .
\url{https://doi.org/10.48550/arXiv.2603.14817}


\bibitem{zaliyev1}
\textbf{Z.S. Aliyev, E.A. Aghayev}, The basis properties of the system of root functions of Sturm-Liouville problem with spectral parameter in the boundary condition, Proceedings of IMM of NAS of Azerbaijan. \textbf{32} (2010), 63-72.

\bibitem{aslanova}
\textbf{Aslanova, N., Aslanov, K. and Kocinac, L}, On Some Spectral Problems for Sturm–Liouville Equation With Operator Coefficients, Math. Meth. Appl. Sci. \textbf{} (2025). \url{https://doi.org/10.1002/mma.10893}

\bibitem{aslanova1}
\textbf{Aslanova, N., Aslanov, K.}, On the Operator Sturm-Liouville Problem with Unbounded Operator Coefficients in Boundary Condition. Journal of Applied and Computational Mechanics, (2025) \url{https://doi.org/10.22055/jacm.2025.48681.5422}

\bibitem{benedek}
\textbf{A.I. Benedek and R. Panzone}, \textit{On inverse eigenvalue problems for a second-order differential
equation with parameter contained in the boundary conditions}, Notas de Algebra y Analisis, INMABB - CONICET, 
Universidad Nacional Del Sur
Bahia Blanca - Argentina, \textbf{9} (1980), 1–13.

\bibitem{binding1}
\textbf{P. A. Binding and P. J. Browne}, \textit{Application of
two parameter eigencurves to Sturm-Liouville problems with
eigenparameter-dependent boundary conditions}, Proc. Roy. Soc.
Edinburgh \textbf{125A} (1995), 1205-1218.


\bibitem{binding2}
\textbf{P. A. Binding, P. J. Browne and B. A. Watson},
\textit{Equivalence of inverse Sturm-Liouville problems with
boundary conditions rationally dependent on the eigenparameter},
J. Math. Anal. Appl. \textbf{291} (2004), 246-261.

\bibitem{chein}
\textbf{Chein-Shan Liu, Yung-Wei Chen, Chih-Wen Chang},
\textit{Precise eigenvalues in the solutions of generalized Sturm–Liouville problems}, Mathematics and Computers in Simulation, 217, 2024, 354-373.
\url{https://doi.org/10.1016/j.matcom.2023.11.008.}

\bibitem{fulton}
\textbf{Charles T. Fulton}, \textit{Two-point boundary value problems with eigenvalue parameter contained in the boundary conditions},Proceedings of the Royal Society of Edinburgh Section A: Mathematics, \textbf{77(3-4)} (1977), 293-308.

\bibitem{guliyev}
\textbf{Guliyev, N.J.}, \textit{Riesz basis criterion for Schrödinger operators with boundary conditions dependent on the eigenvalue parameter}, Anal.Math.Phys. 10, 2 (2020). \url{https://doi.org/10.1007/s13324-019-00348-0}


\bibitem{kerimov}
\textbf{N. B. Kerimov and Y. N. Aliyev}, \textit{The basis
property in $L_p$ of the boundary value problem rationally
dependent on the eigenparameter}, Studia Math. \textbf{174(2)} (2006),
201-212.

\bibitem{kerimov0}
\textbf{N. B. Kerimov}, \textit{Basis Properties in Lp of a Sturm-Liouville Operator with Spectral Parameter in the Boundary Conditions.},  Diff Equat . \textbf{55(2)} (2019),
149–158.

\bibitem{kerimov1}
\textbf{N. B. Kerimov and V. S. Mirzoev}, \textit{On the basis
properties of one spectral problem with a spectral parameter in
boundary conditions}, Siberian Math. J. \textbf{44} (2003), 813-816.

\bibitem{li}
\textbf{Li, Z., Zheng, Z. Zhang, Y.},
\textit{Properties of Eigenvalues and Generalized Eigenfunctions for Sturm–Liouville Problem with Eigenparameter-Dependent Boundary Conditions.}, Mediterr. J. Math. , 22, 222 (2025). \url{https://doi.org/10.1007/s00009-025-02987-z}

\bibitem{liu}
\textbf{Liu, C.-S.; Chang, C.-W.; Kuo, C.-L.}, \textit{Numerical Analysis for Sturm–Liouville Problems with Nonlocal Generalized Boundary Conditions.},Mathematics 2024, 12, 1265. \url{https://doi.org/10.3390/math12081265}

\bibitem{mois}
\textbf{E. I. Moiseev and N. Yu. Kapustin}, \textit{On the singularities of the root space of a spectral problem
with a spectral parameter in the boundary condition}, Dokl. Akad.
Nauk, \textbf{66(1)} (2002), 14-18. 

\bibitem{maris}
\textbf{Maris, E. A., Göktaş, S.}, \textit{On the spectral properties of a Sturm-Liouville problem with eigenparameter in the boundary condition},Hacettepe Journal of Mathematics and Statistics, 49(4), 1373-1382.
\url{https://doi.org/10.15672/hujms.479445}

\bibitem{moller}
\textbf{Möller, M. }, \textit{Minimality of eigenfunctions and associated functions of ordinary differential operators},Adv. Oper. Theory 5, 1014–1025 (2020).
\url{https://doi.org/10.1007/s43036-020-00065-7}


\bibitem{naim}
\textbf{M. A. Naimark}, \textit{Linear differential operators},
2nd edn, Nauka, Moscow, 1969 (in Russian); English trans. of 1st
edn, Parts I,II, Ungar, New York, 1967, 1968.

\bibitem{olgar}
\textbf{H. Olǧar, O. Sh. Mukhtarov, and K. Aydemir. }, \textit{Some Properties of Eigenvalues and Generalized Eigenvectors of One Boundary Value Problem},
Filomat 32, no. 3 (2018): 911–20. \url{https://www.jstor.org/stable/27381776}



\bibitem{rus}
\textbf{E.M. Russakovskii},  Operator treatment of a boundary-value problem with the spectral parameter appearing rationally in the boundary conditions (in Russian). Theory Funct. Funct. Anal. Appl. 30, (1978) 120-128.

\bibitem{shkalikov1}
\textbf{A. A. Shkalikov}, \textit{Boundary value problems for ordinary differential equations with a parameter in the boundary conditions} (Russian. English summary), Tr. Semin. Im. I. G. Petrovskogo \textbf{9} (1983) 190-229.

\bibitem{shkalikov2}
\textbf{A. A. Shkalikov}, \textit{Basis Properties of Root Functions of Differential Operators with Spectral Parameter in the Boundary Conditions}, Diff. Equat., \textbf{55:5} (2019), 631–643.






\end{thebibliography}
\end{document}